\documentclass[12pt,a4paper,oneside]{amsart}

\usepackage{amsthm,amsfonts,amsmath,amssymb,verbatim} 
\usepackage[left=3cm,right=3cm,top=3cm,bottom=3cm]{geometry}
\usepackage[T1]{fontenc}
\usepackage{mathtools}
\usepackage{tikz,graphicx}
\usepackage{tikz-cd}
\usepackage{mathrsfs}
\usepackage{xcolor}
\usepackage{microtype}

\definecolor{modra}{HTML}{0511F2}

\usepackage{fouridx}

\usepackage[normalem]{ulem}

\usepackage{hyperref}

\numberwithin{equation}{section}
\newtheorem{theorem}{Theorem}[section]
\newtheorem{coro}[theorem]{Corollary}
\newtheorem{prop}[theorem]{Proposition}
\newtheorem{lemma}[theorem]{Lemma}
\newcommand{\cha}{\mathcal{X}}
\theoremstyle{definition}
\newtheorem{const}[theorem]{Construction}

\renewcommand{\leq}{\leqslant}
\renewcommand{\geq}{\geqslant}

\newcommand{\xir}[1]{\xi^{(#1)}}

\newcommand{\Mod}[1]{\ (\mathrm{mod}\ #1)}
\newcommand{\rk}{\mathrm{rk}}
\newcommand{\emf}[1]{\textcolor{modra}{\emph{#1}}}

\renewcommand{\o}{\mathrm{o}}
\newcommand{\rs}{\mathrm{rs}}
\newcommand{\so}{\mathrm{so}}
\renewcommand{\sp}{\mathrm{sp}}

\newcommand{\abs}[1]{\lvert#1\rvert}

\newcommand{\la}{\lambda}
\newcommand{\La}{\Lambda}
\DeclareMathOperator{\sgn}{sgn}

\DeclareMathOperator{\ch}{ch}
\newcommand{\mur}[1]{\mu^{(#1)}}
\newcommand{\nur}[1]{\nu^{(#1)}}
\newcommand{\lar}[1]{\lambda^{(#1)}}

\newcommand{\tcore}{t\textup{-core}}
\newcommand{\xcore}[1]{#1\text{-core}}

\DeclareMathOperator{\Ind}{Ind}

\allowdisplaybreaks

\begin{document}

\title{Character factorisations, $z$-asymmetric partitions and plethysm}
\date{}
\author{Seamus Albion}
\address{Fakult\"{a}t f\"{u}r Mathematik, Universit\"{a}t Wien, 
Oskar-Morgenstern-Platz 1, A-1090 Vienna, Austria}
\email{seamus.albion@univie.ac.at}
\thanks{This research was funded in part by the Austrian Science Fund
(FWF) 10.55776/F1002, in the framework of the Special Research Programme
``Discrete Random Structures: Enumeration and Scaling Limits''}


\begin{abstract}
The Verschiebung operators $\varphi_t $ are a family of endomorphisms on the 
ring of symmetric functions, one for each integer $t\geq2$.
Their action on the Schur basis has its origins in work of Littlewood
and Richardson, and is intimately related with the 
decomposition of a partition into its $t$-core and $t$-quotient. 
Namely, they showed that the action on $s_\la$ is zero
if the $t$-core of the indexing partition is nonempty, and otherwise it factors
as a product of Schur functions indexed by the $t$-quotient. 
Much more recently, Lecouvey and, independently, Ayyer and Kumari have 
provided similar formulae for the characters of the symplectic and orthogonal 
groups, where again the combinatorics of cores and quotients plays a
fundamental role.
We embed all of these character factorisations in an infinite family involving
an integer $z$ and parameter $q$ using a very general symmetric function 
defined by Hamel and King.
The proof hinges on a new characterisation of the $t$-cores and $t$-quotients
of $z$-asymmetric partitions
which generalise the well-known classifications for self-conjugate and 
doubled distinct partitions.
We also explain the connection between these results, plethysms
of symmetric functions and characters of the symmetric group.
\end{abstract}

\maketitle

\section{Introduction}

For each integer $t\geq2$ the \emf{Verschiebung operator}\footnote{The name
\emph{Verschiebung} (German for \emph{shift}) comes from the theory of Witt
vectors; see \cite[\S2.9]{Grinberg22} and \cite[Exercise~2.9.10]{GR14}.} 
$\varphi_t$ is an endomorphism on the ring of symmetric functions defined by
\begin{equation}\label{Eq_premik}
\varphi_t h_k = \begin{cases} h_{k/t} & \text{if $t$ divides $k$},\\
0 & \text{otherwise},
\end{cases}
\end{equation}
where $h_k$ denotes the $k$-th complete homogeneous symmetric function.
The action of $\varphi_t$ on the Schur basis was first computed by Littlewood
and Richardson, but phrased in a different way \cite{LR34b,LR35}.
They classified the partitions for which $\varphi_t s_\la=0$ and further show
that when it is nonzero the result is a product of $t$ Schur functions indexed
by partitions depending only on $\la$.
Almost two decades later, Littlewood realised that this action is intimately 
related with the decomposition of a partition into its $t$-core and 
$t$-quotient, concepts which were not yet known at the time of the work
with Richardson.
Much more recently, Lecouvey \cite{Lecouvey09A} and, independently, 
Ayyer and Kumari \cite{AK22} 
computed the action of $\varphi_t$ on the characters of
the symplectic and orthogonal groups in a finite number of variables.
In \cite{Albion23} we lifted these results to the universal characters of
the associated groups.
Again, the combinatorics of cores and quotients is at the heart of the 
evaluations.
Our main result of the present paper, Theorem~\ref{Thm_chiz}, embeds all of 
these ``character factorisations'' in an infinite family paramaterised by an 
integer $z$ and involving a parameter $q$. This is achieved by
computing the action of $\varphi_t$ on a very general symmetric function of 
Hamel and King \cite{HK11a,HK11b}.
For $q=0$ we recover the Schur case and for $z\in\{-1,0,1\}$ the symplectic
and orthogonal cases.
What facilitates this generalisation is a characterisation of the $t$-cores
and $t$-quotients of the $z$-asymmetric partitions of Ayyer and Kumari which
are a $z$-deformation of self-conjugate partitions; see
Theorem~\ref{Thm_zAsym}.
Before explaining our contributions in detail, we survey the history of these
results, since it appears that they are not so well-known.
Moreover, it involves a rich interplay between (modular) representation 
theory, symmetric functions and the combinatorics of integer partitions.

\subsection{Historical background}
The notion of a hook of an integer partition was
introduced by Nakayama in the pair of papers \cite{Nakayama40I,Nakayama40II}.
For an integer $t\geq 2$ he showed that one can associate to each 
partition a \emf{$t$-core}, being a partition containing no hook of length
$t$.
His motivation came from the modular representation theory of the symmetric 
group, and in particular he conjectured that for $t$ prime two partitions
belong to the same $t$-block of the symmetric group if and only if 
they have the same $t$-core \cite[\S6]{Nakayama40II}.
This conjecture was proved several years later by Brauer 
and Robinson \cite{Brauer47,Robinson47}.\footnote{The proof is joint work but
appears in separate papers published simultaneously in the Transactions of
the Royal Society of Canada.}
Following the proof of Nakayama's conjecture, Robinson introduced the notion
of a star diagram associated to a partition, which encodes its 
$t$-hook structure \cite{Robinson48}, work which was continued by 
Staal \cite{Staal50}.
This was independently discovered by Nakayama and Osima, who gave a second,
independent proof of Nakayama's conjecture \cite{NO51}.

Inspired by Robinson's work, Littlewood synthesised the aforementioned ideas
into what he dubbed the $t$-residue and $t$-quotient of a 
partition \cite{Littlewood51}. 
In fact, the $t$-residue is just Nakayama's $t$-core, while the 
$t$-quotient contains the same information as the star diagram of Robinson,
but is more simply constructed. 
Perhaps due to its more straightforward nature, Littlewood's construction is
now the most well-known.
To be a little more explicit, let $\mathscr{P}$ denote the set of partitions
and $\mathscr{C}_t$ the set of all $t$-cores.
What is now known as the \emf{Littlewood decomposition} amounts to a bijection
\begin{align*}
\phi_t : \mathscr{P} &\longrightarrow \mathscr{C}_t\times\mathscr{P}^t \\
\la &\longmapsto \big(\tcore(\la),(\lar{0},\dots,\lar{t-1})\big),
\end{align*}
where $\tcore(\la)$ is Nakayama's $t$-core and the $t$-tuple of partitions
$(\lar{0},\dots,\lar{t-1})$ is Littlewood's $t$-quotient.
The bijection may be realised in several equivalent ways. Below we will use
the realisation in terms of Maya diagrams or, equivalently, the binary
encoding of partitions. Littlewood's original construction was purely
arithmetic, and his motivation was similar to the authors
before him. In his paper he gives a short, independent  proof of Nakayama's 
conjecture, and then uses the $t$-quotient as a tool to produce relationships 
between modular characters inside $t$-blocks.
He gives two further applications of the construction: one to character
values of the symmetric group and one to a particular plethysm of symmetric
functions.

Let $\chi^\la$ denote the irreducible character of the symmetric group
$\mathfrak{S}_n$ indexed by the partition $\la$ of $n$.
We use the usual notations for partitions; see Subsection~\ref{Sec_prelims}
for the relevant definitions.
Here we only note that $t\mu$ stands for the partition with all parts 
multiplied by $t$ and for a partition with empty $t$-core $\sgn_t(\la)$ is 
equal to $\pm1$ and may be defined in terms of the heights of ribbons; 
see \eqref{Eq_sgn}.
Littlewood stated the following theorem.
\begin{theorem}[{\cite[p.~340]{Littlewood51}}]\label{Thm_LittlewoodMult}
Let $\la$ be a partition of $nt$. Then
$\chi^\la(t\mu)=0$ unless the $t$-core of $\la$ is empty, in which case
\begin{equation}\label{Eq_LittlewoodMult}
\chi^\la(t\mu)
=\sgn_t(\la)\Ind_{\mathfrak{S}_{\abs{\lar{0}}}\times\cdots\times
\mathfrak{S}_{\abs{\lar{t-1}}}}^{\mathfrak{S}_n}
\big(\chi^{\lar{0}}\otimes\cdots\otimes\chi^{\lar{t-1}}\big)(\mu).
\end{equation}
\end{theorem}
In fact, this result appears already in a paper of Littlewood and Richardson
from seventeen years prior \cite[Theorem~IX]{LR34b}.\footnote{Curiously,
Littlewood's citation of this result points to his treatise 
\cite{Littlewood40}, although it appears earlier in the work with Richardson.}
There, however, the elegance of the theorem is almost completely obscured by 
the absence of the concepts of the $t$-core and $t$-quotient.
An extension to skew characters $\chi^{\la/\mu}$ was given by Farahat, a 
student of Littlewood \cite{Farahat54}.
For more on this theorem and its generalisations see Subsection~\ref{Sec_chars}.

The second application is to a particular instance of plethysm of symmetric
functions. Again, deferring precise definitions until later on, let
$s_\la=s_\la(x_1,x_2,\dots)$ be the Schur function indexed by $\la$ and 
$p_r(x_1,x_2,\dots)=x_1^r+x_2^r+\cdots$ the $r$-th power sum symmetric
function.
The plethysm $p_r\circ s_\la=s_\la\circ p_r$ is defined by
\begin{equation}\label{Eq_pPlethDef}
s_\la\circ p_r:=s_\la(x_1^r,x_2^r,x_3^r,\ldots).
\end{equation}
Also, for a multiset of skew shapes $\mathcal{S}$ we let 
$c_{\mathcal{S}}^\la$ denote the coefficient of $s_\la$ in the Schur expansion
of $\prod_{\mu\in\mathcal{S}}s_\mu$. When $\mathcal{S}$ consists of only two
straight shapes $\mu,\nu$ then $c_{\mu,\nu}^\la$ are the Littlewood--Richardson
coefficients famously characterised by Littlewood and Richardson in
\cite{LR34a}.
Thus we will refer to the $c_{\mathcal{S}}^\la$ as 
\emf{multi-Littlewood--Richardson coefficients}.
Littlewood's second application is the Schur expansion of the plethysm
\eqref{Eq_pPlethDef}.

\begin{theorem}[{\cite[p.~351]{Littlewood51}}]\label{Thm_SXPLittlewood}
For a partition $\la$ and integer $t\geq 2$,
\[
s_\la\circ p_t
=\sum_{\substack{\nu\\\tcore(\nu)=\varnothing}}
\sgn_t(\nu) c_{\nur{0},\dots,\nur{t-1}}^\la s_\nu.
\]
\end{theorem}

This formula has come to be known as the \emf{SXP rule}.
It has a generalisation as an expansion of the
expression $s_\tau(s_{\la/\mu}\circ p_t)$ due to Wildon \cite{Wildon18} which 
we will meet later on in Section~\ref{Sec_SXP}.

A glance at the structure of the theorems suggests there must be a relation
between them, and indeed the proof of Theorem~\ref{Thm_SXPLittlewood} in 
\cite{Littlewood51} uses Theorem~\ref{Thm_LittlewoodMult}.
Remarkably, Littlewood and Richardson's proof of the first theorem is based 
on a Schur function identity which is in a sense dual to the second theorem.
(Littlewood provides a proof of a slightly more general result in 
\cite{Littlewood51}, of which Theorem~\ref{Thm_LittlewoodMult} is a special 
case, which is independent of the proof given earlier.)
To explain this, recall that the \emf{Hall inner product} is the inner product
on the ring of symmetric functions $\La$ for which the Schur functions are
orthonormal:
\begin{equation}\label{Eq_Hall-Def}
\langle s_\la,s_\mu\rangle=\delta_{\la\mu},
\end{equation}
where $\delta_{\la\mu}$ is the usual Kronecker delta.
As an operator on the algebra of symmetric functions, the plethysm
by $p_t$ has an adjoint, which is denoted by $\varphi_t$.
That is, for any $f,g\in\La$,
\begin{equation}\label{Eq_adjoint}
\langle f\circ p_t,g\rangle = \langle f,\varphi_t g\rangle.
\end{equation}
The operator $\varphi_t$ turns out to be the Verschiebung operator defined
above \eqref{Eq_premik}.

With the aid of Theorem~\ref{Thm_SXPLittlewood} the evaluation of the action
of $\varphi_t$ on the Schur functions is a short exercise.
Setting $(f,g)\mapsto(s_\mu,s_\la)$ in \eqref{Eq_adjoint} gives
$\langle s_\mu\circ p_t,s_\la\rangle =\langle s_\mu,\varphi_t s_\la\rangle$.
Therefore 
\[
\langle s_\mu,\varphi_t s_\la\rangle
=\begin{cases}\sgn_t(\la)c_{\lar{0},\dots,\lar{t-1}}^\mu &\text{if $\tcore(\la)=\varnothing$}, \\
0 & \text{otherwise}.
\end{cases}
\]
By the definition of the multi-Littlewood--Richardson coefficients we obtain
the following.

\begin{theorem}\label{Thm_Littlewoodvarphi}
For $\la$ a partition and $t\geq 2$ an integer we have that 
$\varphi_t s_\la=0$ unless $\tcore(\la)=\varnothing$, in which case
\[
\varphi_t s_\la = \sgn_t(\la)s_{\lar{0}}\cdots s_{\lar{t-1}}.
\]
\end{theorem}

As already mentioned above, this result has its origins in the work of 
Littlewood and Richardson in the 1930's. 
At the generality of the above theorem the result first appears 
in Littlewood's book \cite[\S7.3]{Littlewood40}.
There the language of cores and quotients is of course not used, nor is the
Verschiebung operator. Rather he gives an equivalent formulation in terms of 
his notion of the ``S-function of a series''; see
\cite{Littlewood35} or \cite[Exercise~7.91]{Stanley99}.
There is also an extension of Theorem~\ref{Thm_Littlewoodvarphi} to skew Schur
functions due to Farahat and Macdonald; see Theorem~\ref{Thm_skewSchur} below.
How precisely  Theorems~\ref{Thm_LittlewoodMult}, \ref{Thm_SXPLittlewood} and \ref{Thm_Littlewoodvarphi} are equivalent will be explained in 
Subsection~\ref{Sec_chars}.

\subsection{Generalisations to classical group characters}
In their paper on what are now called LLT polynomials, 
Lascoux, Leclerc and Thibon pointed out that the adjoint 
relationship \eqref{Eq_adjoint} combined with a refinement of the Littlewood
decomposition to ribbon tableaux due to Stanton and White \cite{SW85} leads to 
a combinatorial proof of the identity of Theorem~\ref{Thm_Littlewoodvarphi} 
\cite[\S IV]{LLT97}.
Indeed, the operator $\varphi_t$ and its plethysm adjoint are ``$q$-deformed''
and then used to define the LLT polynomials.
In extending this construction to other types, Lecouvey proved 
beautiful variations of Theorem~\ref{Thm_Littlewoodvarphi} for the
characters of $\mathrm{Sp}_{2n}$ and
$\mathrm{O}_{2n}$ in the case $t$ is odd
and $\mathrm{SO}_{2n+1}$ for general $t$ \cite[\S3]{Lecouvey09B}.
(Here and throughout all matrix groups are taken over $\mathbb{C}$.)
Rather than expressing these results as products of characters, he gives the 
expansion of the evaluation in terms of Weyl characters
where the coefficients are branching coefficients corresponding to the
restriction of an irreducible polynomial representation to a subgroup of
Levi type.
The obstruction for $t$ even in the first two cases is precisely that the
coefficients cannot be interpreted as branching coefficients.

Recently Ayyer and Kumari rediscovered the factorisation results of Lecouvey,
but in a slightly different form by ``twisting'' a finite set of $n$ variables
by a primitive $t$-th root of unity \cite{AK22}.
This point of view is explained in Section~\ref{Sec_disc}.
By working with the explicit Laurent polynomial expressions 
for the symplectic and orthogonal characters they could show that for all 
$t\geq 2$ these twisted characters factor as a product of other characters.
They also characterise the vanishing of these twisted characters in a much 
simpler manner. For example they show that the twisted character of 
$\mathrm{SO}_{2n+1}$ indexed by $\la$ vanishes unless 
$\tcore(\la)$ is self-conjugate.
The even orthogonal and symplectic cases admit similarly simple descriptions.
For $t=2$ these factorisations may be found already in the work of 
Mizukawa \cite{Mizukawa03}.

Lecouvey also proved striking extensions of Theorem~\ref{Thm_SXPLittlewood}
to the universal characters of the symplectic and orthogonal groups
\cite{Lecouvey09A}.
These are symmetric function lifts of the ordinary characters first defined by 
Koike and Terada \cite{KT87} using the Jacobi--Trudi formulae of Weyl.
(Lecouvey's extensions are anticipated by work of Littlewood 
\cite{Littlewood58} for the ordinary characters and Scharf and Thibon for the 
universal characters \cite[\S6]{ST94}, both only for $t=2,3$.)
Inspired by the work of Ayyer and Kumari we lifted their 
factorisations to the level of universal characters 
\cite{Albion23}.\footnote{At the time we were unfortunately not aware of the 
work of Lecouvey.}
Our proofs there are based on the Jacobi--Trudi formulae for these
symmetric functions.
In the present work we utilise a different approach based on 
expressions for the universal characters in terms of skew Schur functions.
For example, let $\so_\la$ denote the universal odd orthogonal character.
Then
\begin{equation}\label{Eq_so-def}
\so_\la 
:=\det_{1\leq i,j\leq l(\la)}(h_{\la_i-i+j}+h_{\la_i-i-j-1})
= \sum_{\substack{\mu\in\mathscr{P}_0\\\mu\subseteq\la}}
(-1)^{(\abs{\mu}-\rk(\mu))/2}s_{\la/\mu},
\end{equation}
where $\mathscr{P}_0$ is the set of self-conjugate partitions,
$\mu\subseteq\la$ means the Young diagram of $\mu$ is contained in that of 
$\la$ and $\rk(\mu)$ denotes the Frobenius rank of $\mu$.
We will now state the expression for $\varphi_t\so_\la$,
in which we will write $\tilde\la:=\tcore(\la)$, a short-hand also used below
whenever it is convenient.
We also note that for a pair of partitions $\la,\mu$ the symmetric function
$\rs_{\la,\mu}$ is the universal character lift of the irreducible 
rational representation of $\mathrm{GL}_n$ indexed by the pair of 
partitions $(\la,\mu)$;
see \eqref{Eq_rsDef} and the surrounding discussion for a definition.

\begin{theorem}\label{Thm_so}
For $\la$ a partition and $t\geq 2$ an integer we have that 
$\varphi_t \so_\la=0$ unless $\tcore(\la)$ is self-conjugate, in which case
\[
\varphi_t\so_\la
=(-1)^{(\abs{\tilde\la}-\rk(\tilde\la))/2}\sgn_t(\la/\tilde\la)
\prod_{r=0}^{\lfloor (t-2)/2\rfloor}
\rs_{\lar{r},\lar{t-r-1}}\times\begin{cases}
1 & \text{$t$ even}, \\
\so_{\lar{(t-1)/2}}  & \text{$t$ odd}.
\end{cases}
\]
\end{theorem}

This may be found in various forms in \cite[\S3.2.4]{Lecouvey09B},
\cite[Theorem~2.17]{AK22} and \cite[Theorem~3.4]{Albion23}.
The key difference between this theorem and all its previous iterations is
that the overall sign is explicitly expressed in terms of 
statistics on $\la$ and its $t$-core.
While not visible from the above we are also able to show that in the 
symplectic and even orthogonal cases the sign is just as simple.
The proof of the above theorem we present below uses the skew Schur expansion in
\eqref{Eq_so-def}, the skew Schur function case of
Theorem~\ref{Thm_Littlewoodvarphi} (Theorem~\ref{Thm_skewSchur} below)
and properties of the Littlewood decomposition restricted to the set of 
self-conjugate partitions.
More precisely, it was observed by
Osima \cite{Osima52} that a partition is self-conjugate if and only if
$\tcore(\la)$ is self-conjugate and
\begin{equation}\label{Eq_IntroReflect}
\lar{r}=(\lar{t-r-1})' \quad\text{for $0\leq r\leq t-1$}.
\end{equation}
Note that the partitions paired by this condition are precisely the partitions
paired in the factorisation of $\varphi_t\so_\la$.

In fact much more is true. 
Our main result, which we state as Theorem~\ref{Thm_chiz} 
below,
embeds Theorem~\ref{Thm_so} as the $(z,q)=(0,1)$ case of an infinite family 
of such factorisations where $z$ is an arbitrary integer and $q$ is a formal
variable. 
The generalisation of the character $\so_\la$, denoted
$\cha_\la(z;q)$, is a 
symmetric function defined by Hamel and King \cite{HK11a,HK11b}, building on
work of Bressoud and Wei \cite{BW92}.
It may be expressed as a Jacobi--Trudi-type determinant or as a sum  
of skew Schur functions \`a la \eqref{Eq_so-def}.
This sum is indexed by $z$-asymmetric partitions, a 
term coined by Ayyer and Kumari, which are a $z$-deformation of 
self-conjugate partitions.
In fact, what facilitates the factorisation of this object under 
$\varphi_t$ is that the Littlewood decomposition for $z$-asymmetric partitions
has a nice structure, involving ``conjugation conditions'' such as
\eqref{Eq_IntroReflect}.
Indeed, this is our other main result, Theorem~\ref{Thm_zAsym},
which characterises $z$-asymmetric partitions in terms of their Littlewood
decompositions.
For $z=0$ this is the self-conjugate case discussed prior, and for $z=1$ this
appears in the seminal work of Garvan, Kim and Stanton on 
cranks \cite{GKS90}.

\subsection{Summary of the paper}
The paper reads as follows. In the next section we introduce the necessary
definitions and conventions for integer partitions, including the 
Littlewood decomposition. This includes our first main result,
Theorem~\ref{Thm_zAsym}, the characterisation of $z$-asymmetric partitions
under the Littlewood decomposition.
Section~\ref{Sec_Symm} then turns to symmetric functions and universal 
characters. We survey the action of the Verschiebung operators on the classical
bases of the ring of symmetric functions, and introduce a new deformation of
the rational universal characters which arise naturally in our main 
factorisation theorem.
Section~\ref{Sec_proofs} then contains the companions of Theorem~\ref{Thm_so}
for the symplectic and even orthogonal characters, our generalisation,
stated as Theorem~\ref{Thm_chiz}, and its proof.
Then Section~\ref{Sec_SXP} is used to survey the known SXP rules for
Schur functions and universal characters.
This includes Wildon's generalisation of Theorem~\ref{Thm_SXPLittlewood} which 
we show is equivalent to the skew case of Theorem~\ref{Thm_Littlewoodvarphi}
(Theorem~\ref{Thm_skewSchur} below).
Using our combinatorial setup, we give reinterpretations of Lecouvey's 
SXP rules, and in particular show that for all types they may be
expressed as sums over partitions with empty $t$-core.
We close with some remarks about related results, including a discussion of 
the precise relationship between the first three theorems of the introduction.

\section{Partitions and the Littlewood decomposition}
This section contains the necessary preliminaries regarding integer partitions. 
We also describe the Littlewood decomposition in terms of Maya diagrams
which is essentially the abacus model of James and Kerber \cite{JK81}.
Our main results in this section,
Theorem~\ref{Thm_zAsym} and its corollaries, give a characterisation of 
$z$-asymmetric partitions in terms of the Littlewood decomposition.

\subsection{Preliminaries}\label{Sec_prelims}
A \emf{partition} is a weakly decreasing sequence of nonnegative integers 
$\la=(\la_1,\la_2,\la_3,\dots)$ such that the \emf{size}
$\abs{\la}:=\la_1+\la_2+\la_3+\cdots$ is finite.
The nonzero $\la_i$ are called \emf{parts} and the number of parts the
\emf{length}, denoted $l(\la)$.
The set of all partitions is written $\mathscr{P}$ and the empty partition,
the unique partition of $0$, is denoted by $\varnothing$.
We write $(m^\ell)$ for the partition with $\ell$ parts equal to $m$, and the 
sum $\la+(m^\ell)$ is then the partition obtained by adding 
$m$ to the first $\ell$ parts of $\la$.
We identify a partition with its \emf{Young diagram},
which is the left-justified array of cells consisting of $\la_i$ cells in 
row $i$ with $i$ increasing downward.
An example is given in Figure~\ref{Fig_YD}.
We define the \emf{conjugate} partition $\la'$ by reflecting the diagram of 
$\la$ in the main diagonal, so that the conjugate of $(6,5,5,1)$ is
$(4,3,3,3,3,1)$.
A partition is \emf{self-conjugate} if $\la=\la'$.

The \emf{Frobenius rank} of a partition, $\rk(\la)$, is defined as the number 
of cells along the main diagonal of its Young diagram.
We extend this by an integer $c\in\mathbb{Z}$ to a statistic $\rk_c(\la)$
which we call the \emf{$c$-shifted Frobenius rank}.
If $c\geq0$ this is the Frobenius rank of the partition obtained
by deleting the first $c$ rows of $\la$, while for $c\leq 0$ it is the Frobenius
rank of the partition with the first $-c$ columns of $\la$ removed.
Another way to notate partitions is with \emf{Frobenius notation}, which
records the number of cells to the right of and below each cell on this 
main diagonal. This is written 
\[
\la = \big(\la_1-1,\dots,\la_{\rk(\la)}-\rk(\la)~\vert~\la_1'-1,
\dots,\la_{\rk(\la)}'-\rk(\la)\big);
\]
again, see Figure~\ref{Fig_YD} for an example.
Any two strictly decreasing nonnegative integer sequences $u,v$ with the same
number of elements, say $k$, thus give a unique partition $\la=(u~\vert~v)$ 
of Frobenius rank $k$. Clearly self-conjugate partitions are those of the 
form $(u~\vert~u)$. 
Now let $u+z:=(u_1+z,\dots,u_k+z)$ for any 
$z\in\mathbb{Z}$. Ayyer and Kumari define 
\emf{$z$-asymmetric partitions} to be those of the form $(u+z~\vert~u)$ for 
any sequence $u$ (of any length) and fixed $z\in\mathbb{Z}$ 
\cite[Definition~2.9]{AK22}.
The set of $z$-asymmetric partitions is denoted by $\mathscr{P}_z$
and $(6,5,5,1)$ in Figure~\ref{Fig_YD} is $2$-asymmetric.
The generating function for $z$-asymmetric partitions is given by
\[
\sum_{\la\in\mathscr{P}_z}q^{\abs{\la}}=(-q^{1+\abs{z}};q^2)_\infty.
\]
This is easy to see by the fact that a $z$-asymmetric partition is uniquely
determined by its set of hook lengths on the main diagonal.
These are all distinct integers of the form ``odd plus $\abs{z}$'', which gives
the proof.
Clearly the conjugate of a $z$-asymmetric partition is $-z$-asymmetric.

\begin{figure}[htb]
\centering
\begin{tikzpicture}[scale=.5]
\foreach \i [count=\ii] in {6,5,5,1}
\foreach \j in {1,...,\i}{\draw (\j,1-\ii) rectangle (\j+1,-\ii);}
\foreach \i in {1,2,3}{
\draw[fill=gray] (\i,1-\i) rectangle (\i+1,-\i);}
\foreach \i [count=\ii] in {6,5,5,1}
\foreach \j in {1,...,\i}{\draw (\j+11,1-\ii) rectangle (\j+12,-\ii);}
\node at (12.5,-0.5) {$9$}; \node at (13.5,-0.5) {$7$}; \node at (14.5,-0.5) {$6$}; \node at (15.5,-0.5) {$5$}; \node at (16.5,-0.5) {$4$}; \node at (17.5,-0.5) {$1$};
\node at (12.5,-1.5) {$7$}; \node at (13.5,-1.5) {$5$}; \node at (14.5,-1.5) {$4$}; \node at (15.5,-1.5) {$3$}; \node at (16.5,-1.5) {$2$}; 
\node at (12.5,-2.5) {$6$}; \node at (13.5,-2.5) {$4$}; \node at (14.5,-2.5) {$3$}; \node at (15.5,-2.5) {$2$}; \node at (16.5,-2.5) {$1$};
\node at (12.5,-3.5) {$1$};
\end{tikzpicture}
\caption{The partition $\la=(6,5,5,1)=(5,3,2~\vert~3,1,0)$ with its main diagonal 
shaded (left) and the same partition with hook length of each cell 
inscribed (right). We have $\abs{\la}=17$, $l(\la) = 4$, $\rk(\la)=3$,
$\rk_2(\la)=1$ and $\rk_{-3}(\la)=2$.}
\label{Fig_YD}
\end{figure}
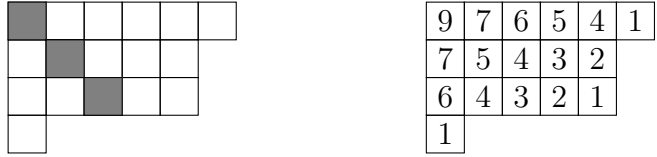

Given a cell $s$ in the Young diagram of $\la$ its \emf{hook length} is one 
more than the sum of the number of cells below and to the right of $s$;
see Figure~\ref{Fig_YD}. The \emf{hook} of $s$ is then the set of cells counted.
A hook is a \emf{principal hook} if it is the hook of a cell on the main 
diagonal.
For an integer $t\geq 2$ we say a partition is 
a \emf{$t$-core} if it contains no cell with hook length $t$ 
(or, equivalently, no cell with hook length divisible by $t$).
For a pair of partitions $\la$, $\mu$ we say $\mu$ is \emf{contained} in 
$\la$, written $\mu\subseteq\la$, if its Young diagram may be drawn inside the 
Young diagram of $\la$.
The corresponding \emf{skew shape} is the arrangement of cells formed by
removing $\mu$'s diagram from $\la$'s.
A skew shape is a \emf{ribbon} if it is edge-connected and contains no
$2\times 2$ square of cells, and a \emf{$t$-ribbon} is a ribbon containing
$t$ cells.\footnote{Elsewhere in the literature ribbons are variously called
\emf{border strips}, \emf{rim hooks} or \emf{skew hooks}.}
The \emf{height} of a ribbon $R$, $\mathrm{ht}(R)$, is one less than the 
number of rows it occupies; see Figure~\ref{Fig_ribbon}.

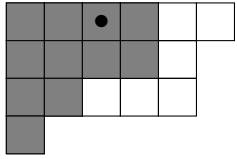
\begin{figure}[htb]
\centering
\begin{tikzpicture}[scale=.5]
\foreach \i [count=\ii] in {6,5,5,1}
\foreach \j in {1,...,\i}{\draw (\j,1-\ii) rectangle (\j+1,-\ii);}
\foreach \i [count=\ii] in {4,4,2,1}
\foreach \j in {1,...,\i}{\draw[fill=gray] 
(\j,1-\ii) rectangle (\j+1,-\ii);}
\node at (3.5,-0.5) {$\bullet$};
\end{tikzpicture}
\caption{The pair of partitions $(4,4,2,1)\subseteq(6,5,5,1)$.
The unshaded cells form a $6$-ribbon of height $2$ and the corresponding
cell with hook length $6$ is marked.}
\label{Fig_ribbon}
\end{figure}

We say a skew shape $\la/\mu$ is \emf{$t$-tileable} if there exists
a sequence of partitions
\[
\mu=:\nur{0}\subseteq\nur{1}\subseteq\cdots\subseteq\nur{m-1}
\subseteq\nur{m}:=\la
\]
such that the skew shapes $\nur{r}/\nur{r-1}$ are each $t$-ribbons for 
$1 \leq r\leq m$.
It is a non-trivial fact, see, e.g.~\cite[Lemma~4.1]{Pak00},
that the sign
\begin{equation}\label{Eq_sgn}
\sgn_t(\la/\mu):=(-1)^{\sum_{r=1}^m\mathrm{ht}(\nur{r}/\nur{r-1})}
\end{equation}
is constant over the set of all $t$-ribbon decompositions of $\la/\mu$
(so, indeed, the above is well-defined).
In the case $\mu=\varnothing$ and $t=2$ the above sign is simply equal to
\[
\sgn_2(\la)=(-1)^{\mathrm{odd}(\la)/2}
\]
where $\mathrm{odd}(\la)$ is equal to the number of odd parts of $\la$;
see, e.g., \cite[Equation~(5.15)]{BR01}.

\subsection{Littlewood's decomposition}
Here we describe the Littlewood decomposition through the lens of 
Maya diagrams, which is essentially the \emf{abacus} of James and Kerber
\cite[\S2.7]{JK81} or the \emf{Brettspiele} of Kerber, S\"anger and Wagner
\cite{KSW81}. 
Littlewood's original algebraic description may be found in 
\cite{Littlewood51} and \cite[p.~12]{Macdonald95}.

Given a partition $\la$ its \emf{beta set} is the subset of the half integers 
given by
\[
\beta(\la):=\Big\{\la_i-i+\frac{1}{2} : i\geq 1\Big\}.
\]
This is visualised as a configuration of ``beads'' on the real line placed at 
the positions indicated by $\beta(\la)$, and this visualization is the
\emf{Maya diagram}.
Note that for any partition the configuration will eventually contain only 
beads to the left and only empty spaces to the right.
The map from partitions to Maya diagrams is clearly a bijection,
and one way to reconstruct $\la$ from $\beta(\la)$ is to count the number
of empty spaces to the left of each bead starting from the right.
From the Maya diagram we extract $t$ subdiagrams, called 
\emf{runners}, formed by the beads
at positions $x$ such that $x-1/2$ is equal to $r$ modulo $t$ for 
$0\leq r\leq t-1$. Arranging the runners with $r$ increasing upward 
we obtain the \emf{$t$-Maya diagram}. 
An example of this procedure is given in Figure~\ref{Fig_maya}.
The partitions corresponding to each runner are denoted by $\lar{r}$ according 
to the
residues modulo $t$ of the original positions, and these precisely form 
Littlewood's \emf{$t$-quotient}.

The next important observation is that $t$-hooks in $\la$ correspond
to beads in its $t$-Maya diagram which contain no bead immediately to their 
left.
For example, Figure~\ref{Fig_YD} shows that $(6,5,5,1)$ contains two 
$3$-hooks, and in Figure~\ref{Fig_maya} one bead in runner $0$ and one in 
runner $2$ have free spaces to their left.
Moving such a bead one space to its left removes the $t$-ribbon associated with 
that hook.
Repeating this procedure until all beads are flush-left in the $t$-Maya diagram
produces a unique partition $\tcore(\la)$ which, as the notation suggests,
is a $t$-core.
The uniqueness is clear from the $t$-Maya diagram picture.
Furthermore, the height of the removed ribbon is equal to the number of beads
between its initial and terminal position, i.e., to
$\abs{\beta(\la)\cap \{x-1,\dots,x-t+1\}}$ if we move the bead at position $x$.
Note that in the ordinary Maya diagram this corresponds to the number of beads
``jumped over''.
Let us collect these observations into the following theorem.

\begin{figure}[htb]
\centering
\begin{tikzpicture}[scale=0.7]
\foreach \i in {0,-2,-3,-4}{
\draw[<->,opacity=0.7] (-7.5,\i) -- (7.5,\i);}
\foreach \i in {-7,...,-1,1,2,3,4,5,6,7}{
\foreach \j in {0,-2,-3,-4}{
\draw[opacity=0.5] (\i,\j+0.1) -- (\i,\j-0.1);}}
\foreach \i in {-4,-2,-1,0,1,4,6}{
\draw[thick,fill=white] (\i+0.5,0) circle (5pt);}
\foreach \i in {-6,-3,3}{
\draw[thick,fill=red] (\i+0.5,0) circle (5pt);}
\foreach \i in {-5}{
\draw[thick,fill=blue] (\i+0.5,0) circle (5pt);}
\foreach \i in {-7,2,5}{
\draw[thick,fill=green] (\i+0.5,0) circle (5pt);}
\draw[opacity=0.5,dashed] (0,0.5) -- (0,-0.5);
\foreach \i in {-2,-1,2,3,4,5,6}{
\draw[thick,fill=white] (\i+0.5,-2) circle (5pt);}
\foreach \i in {-1,0,1,2,3,4,5,6}{
\draw[thick,fill=white] (\i+0.5,-3) circle (5pt);}
\foreach \i in {0,2,3,4,5,6}{
\draw[thick,fill=white] (\i+0.5,-4) circle (5pt);}
\foreach \i in {-7,-6,-5,-4,-3,-2,-1,1}{
\draw[thick,fill=red] (\i+0.5,-4) circle (5pt);}
\foreach \i in {-7,-6,-5,-4,-3,-2}{
\draw[thick,fill=blue] (\i+0.5,-3) circle (5pt);}
\foreach \i in {-7,-6,-5,-4,-3,0,1}{
\draw[thick,fill=green] (\i+0.5,-2) circle (5pt);}
\node at (-8,-4) {$\lar{0}$}; \node at (-8,-3) {$\lar{1}$};
\node at (-8,-2) {$\lar{2}$}; \node at (-8,0) {$\la$};
\draw[opacity=0.5,dashed] (0,-1.5) -- (0,-4.5);
\draw[thick,->] (0,-0.6) -- (0,-1.4);
\end{tikzpicture}
\caption{The Maya diagram of $\la=(6,5,5,1)$ (top) and the $3$-Maya diagram of
the same partition (bottom). We have that
$3$-core$(\la)=(1,1)$, $\kappa_3((1,1))=(1,-1,0)$ and
 $(\lar{0},\lar{1},\lar{2})=((1),\varnothing,(2,2))$.}
\label{Fig_maya}
\end{figure}
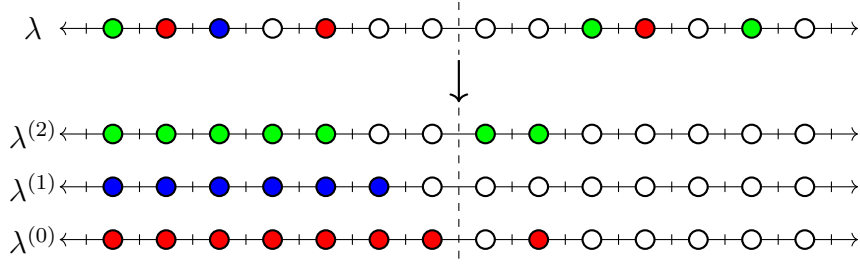

\begin{theorem}[Littlewood's decomposition]
For any integer $t\geq 2$ the above procedure encodes a bijection
\begin{align*}
\mathscr{P}&\longrightarrow \mathscr{C}_t\times\mathscr{P}^t \\
\la &\longmapsto\big(\tcore(\la),(\lar{0},\dots,\lar{t-1})\big)
\end{align*}
such that $\abs{\la}=\abs{\tcore(\la)}+t(\abs{\lar{0}}+\cdots+\abs{\lar{t-1}})$.
\end{theorem}

When a skew shape $\la/\mu$ is $t$-tileable can be characterised completely
in terms of the Littlewood decomposition of $\la$ and $\mu$.
Since $\la/\mu$ being $t$-tileable means that we may obtain the diagram of
$\mu$ from that of $\la$ by removing ribbons in any order, it
follows that $\la/\mu$ is $t$-tileable if and only if $\tcore(\la)=\tcore(\mu)$
and $\mur{r}\subseteq\lar{r}$ for each $0\leq r\leq t-1$.

We will also need a different characterisation of $t$-cores.
Call a Maya diagram \emf{balanced} if it contains as many beads to the
right of $0$ as empty spaces to the left.
The way we defined Maya diagrams ensures they are always balanced, but
Figure~\ref{Fig_maya} shows that the constituent diagrams of the quotient 
need not be.
Let $c_r^+$ (resp.~$c_r^-$) denote the number of beads to the right of 
$0$ (resp.~number of empty spaces to the left of $0$) in row $\lar{r}$ of
the $t$-Maya diagram.
Now the sequence of integers $(c_0,\dots,c_{t-1})$ defined by
$c_r:=c_r^+-c_r^-$ has total sum zero, and is invariant under valid
bead movements.
As observed by Garvan, Kim and Stanton, this encodes a bijection
\cite[Bijection 2]{GKS90}
\begin{equation}\label{Eq_code}
\kappa_t:\mathscr{C}_t\longrightarrow
\{(c_0,\dots,c_{t-1})\in\mathbb{Z}^t : c_0+\cdots+c_{t-1}=0\}
\end{equation}
such that for $\mu\in\mathscr{C}_t$
\[
\abs{\mu} = \sum_{r=0}^{t-1}\bigg(\frac{tc_r^2}{2}+rc_r\bigg).
\]
In what follows we extend \eqref{Eq_code} to a map
$\mathscr{P}\longrightarrow\mathbb{Z}^t$, the fibres of which are the sets
of all partitions with a given core.

In the introduction we noted that self-conjugate partitions satisfy a nice
symmetry with respect to the Littlewood decomposition.
To explain where this comes from, note that
the conjugate of a partition can be read off its (ordinary) Maya diagram by 
interchanging beads and empty spaces and then reflecting the picture about $0$.
In the $t$-Maya diagram this corresponds to conjugating each runner and 
reversing the order of the runners.
This implies that the $t$-quotient of $\la'$ is given by
$((\lar{t-1})',\dots,(\lar{0})')$ in terms of the $t$-quotient of $\la$.
Furthermore, we have that
$\tcore(\la')=\tcore(\la)'$ which, if $\kappa_t(\la)=(c_0,\dots,c_{t-1})$,
translates to $\kappa_t(\la')=(-c_{t-1},\dots,-c_0)$
in terms of \eqref{Eq_code}.
From these properties it immediately follows that the Littlewood decomposition
of a self-conjugate partition much satisfy  
$c_r+c_{t-r-1}=0$ for $0\leq r\leq t-1$ and
$\lar{r}=(\lar{t-r-1})'$ for $r$ in the same range.
This is equivalent to the conditions given in the introduction.
Garvan, Kim and Stanton \cite[\S8]{GKS90} show that something similar holds 
for $1$-asymmetric partitions.
\begin{prop}
If $\la\in\mathscr{P}_1$ then $\tcore(\la),\lar{0}\in\mathscr{P}_1$ and
the remaining entries in the quotient satisfy
$\lar{r}=(\lar{t-r})'$ for $1\leq r\leq t-1$.
\end{prop}

Our first main result is a generalisation of this proposition to 
$z$-asymmetric partitions.
To fix some notation, 
let $\mathcal{C}_{z;t}\subset\mathbb{Z}^t$ consist of those 
$t$-tuples for which $c_{r}+c_{z-r-1}=0$ for $0\leq r\leq z-1$ and 
$c_s+c_{t+z-s-1}=0$ for $z\leq s\leq t-1$.
Also recall the $c$-shifted Frobenius rank $\rk_c(\la)$ from the previous
Subsection~\ref{Sec_prelims}.

\begin{theorem}\label{Thm_zAsym}
Let $t\geq 2$ and $z$ be integers and $\la$ a partition such that
$0\leq z\leq t-1$ and $\la\in\mathscr{P}_z$.
Then $\kappa_t(\tcore(\la))\in \mathcal{C}_{z;t}$ and the quotient
$(\lar{0},\dots,\lar{t-1})$ is such that 
for $0\leq r\leq z-1$ with $c_r\geq0$ there exist partitions
$\nur{r}$ with 
\begin{subequations}\label{Eq_folding}
\begin{equation}
\lar{r}=\nur{r}+(1^{c_r+\rk_{c_r}(\nur{r})})\quad\text{and}\quad
\lar{z-r-1}=(\nur{r})' + (1^{\rk_{c_r}(\nur{r})}) \label{Eq_zFold}
\end{equation}
and for $z\leq s\leq t-1$,
\begin{equation}
\lar{s}=(\lar{t+z-s-1})'\label{Eq_tFold}.
\end{equation}
\end{subequations}
\end{theorem}

\begin{proof}
The proof is by induction on $z$.
For $z=0$ the result is clear from the properties of self-conjugate
partitions under the Littlewood decomposition. 
Now choose a strict partition $v$ and let $\la=(v+z-1~\vert~v)$ for 
some fixed $z\geq 1$. Assume that $\kappa_t(\tcore(\la))
\in\mathcal{C}_{z-1;t}$ and further that the conditions \eqref{Eq_folding}
are satisfied (with $z$ replaced by $z-1$ in the latter).
We wish to show that the partition $\mu=(v+z~\vert~v)$ 
has $\kappa_t(\tcore(\mu))\in\mathcal{C}_{z;t}$ and
that the conditions \eqref{Eq_folding} hold for $\mu$.
Also set $\kappa_t(\tcore(\la))=(c_1,\dots,c_{t-1})$ and
$\kappa_t(\tcore(\mu))=(d_1,\dots,d_{t-1})$.

The key observation is that we may obtain the $t$-Maya diagram of $\mu$ from
that of $\la$ as follows: beads lying at positive positions are moved 
upwards cyclically one runner in the same column, except those passing
from $\lar{t-1}$ to $\lar{0}$, which move an additional space to the right.
An example of this is given in Figure~\ref{Fig_mayaCut}.
If we imagine that the $t$-Maya diagram is wrapped around a bi-infinite 
cylinder, then
this corresponds to cutting the cylinder along $0$, ``twisting'' so that
beads passing from $r=t-1$ to $r=0$ are also moved one space to the
right, and then re-gluing.
From this construction we observe that for $0\leq r\leq z-1$
\begin{equation}\label{Eq_dsum}
d_r + d_{z-r-1} = c_{r-1}^+-c_r^-+c_{z-r-2}^+-c_{z-r-1}^-.
\end{equation}
We already have that $c_r+c_{z-r-2}=0$ for $0\leq r\leq z-2$ by our assumption.
However, \eqref{Eq_zFold} implies the slightly stronger condition
that $c_r^-=c_{z-r-2}^+$ (or, equivalently, $c_r^+=c_{z-r-2}^-$).
This is because conjugation of a runner interchanges $c_r^+$ and $c_r^-$.
Thus, in the range $1\leq r\leq z-2$ we have that \eqref{Eq_dsum} vanishes.
For $r=0$ one needs to use that $\lar{z}=(\lar{t-1})'$ and
$c_z+c_{t-1}=0$.
The same argument in the range $z\leq s\leq t-1$ completes the proof that 
$\kappa_t(\tcore(\mu))\in\mathcal{C}_{z;t}$.
\begin{figure}[htb]
\centering
\begin{tikzpicture}[scale=0.7]
\foreach \i in {0,-1,-2,-4,-5,-6}{
\draw[<->,opacity=0.7] (-7.5,\i) -- (7.5,\i);}
\foreach \i in {-7,...,-1,1,2,3,4,5,6,7}{
\foreach \j in {0,-1,-2,-4,-5,-6}{
\draw[opacity=0.5] (\i,\j+0.1) -- (\i,\j-0.1);}}
\draw[opacity=0.5,dashed] (0,0.5) -- (0,-2.5);
\foreach \i in {-2,-1,2,3,4,5,6}{
\draw[thick,fill=white] (\i+0.5,0) circle (5pt);}
\foreach \i in {-1,0,1,2,3,4,5,6}{
\draw[thick,fill=white] (\i+0.5,-1) circle (5pt);}
\foreach \i in {0,2,3,4,5,6}{
\draw[thick,fill=white] (\i+0.5,-2) circle (5pt);}
\foreach \i in {-7,-6,-5,-4,-3,-2,-1,1}{
\draw[thick,fill=red] (\i+0.5,-2) circle (5pt);}
\foreach \i in {-7,-6,-5,-4,-3,-2}{
\draw[thick,fill=blue] (\i+0.5,-1) circle (5pt);}
\foreach \i in {-7,-6,-5,-4,-3,0,1}{
\draw[thick,fill=green] (\i+0.5,0) circle (5pt);}
\foreach \i in {-2,-1,0,1,2,3,4,5,6}{
\draw[thick,fill=white] (\i+0.5,-4) circle (5pt);}
\foreach \i in {-1,0,2,3,4,5,6}{
\draw[thick,fill=white] (\i+0.5,-5) circle (5pt);}
\foreach \i in {0,3,4,5,6}{
\draw[thick,fill=white] (\i+0.5,-6) circle (5pt);}
\foreach \i in {-7,-6,-5,-4,-3,-2,-1,1,2}{
\draw[thick,fill=red] (\i+0.5,-6) circle (5pt);}
\foreach \i in {-7,-6,-5,-4,-3,-2,1}{
\draw[thick,fill=blue] (\i+0.5,-5) circle (5pt);}
\foreach \i in {-7,-6,-5,-4,-3}{
\draw[thick,fill=green] (\i+0.5,-4) circle (5pt);}
\node at (-8,-2) {$\lar{0}$}; \node at (-8,-1) {$\lar{1}$};
\node at (-8,0) {$\lar{2}$};
\node at (-8,-6) {$\mur{0}$}; \node at (-8,-5) {$\mur{1}$};
\node at (-8,-4) {$\mur{2}$};
\draw[opacity=0.5,dashed] (0,-3.5) -- (0,-6.5);
\draw[thick,->] (0,-2.6) -- (0,-3.4);
\end{tikzpicture}
\caption{The $3$-Maya diagram of $(6,5,5,1)$ (top) and the $3$-Maya diagram
of $(7,6,6,1)$ (bottom) corresponding to action the ``cut and twist'' map.}
\label{Fig_mayaCut}
\end{figure}
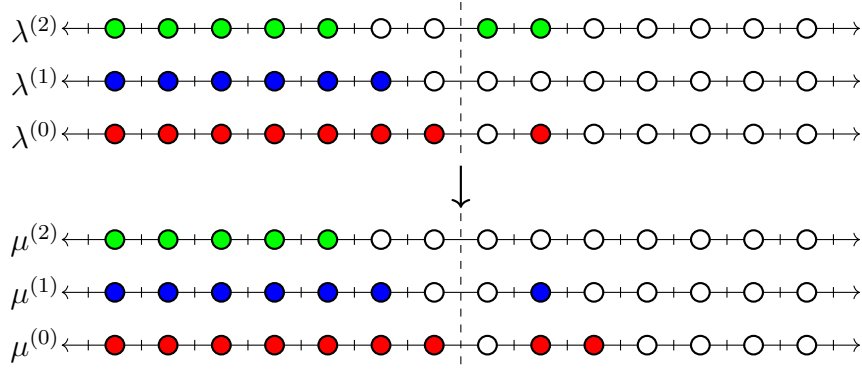

Let $\lar{r}_{<0}$ (resp.~$\lar{r}_{>0}$) denote the negative 
(resp.~positive) half of the runner corresponding to $\lar{r}$.
It is clear that, when indices are read modulo $t$,
$\mur{r}=\lar{r}_{<0}\cup\lar{r-1}_{>0}$ and
$\mur{z-r-1}=\lar{z-r-1}_{<0}\cup\lar{z-r-2}_{>0}$ for $0\leq r\leq z-1$.
Since the $\lar{r}$ satisfy $\eqref{Eq_folding}$ in the range $1\leq r\leq z-2$,
then so will the $\mur{r}$ with $c_r$ replaced by $d_r$.
The cases $r=0$ and for $z\leq s\leq t-1$ follow by the same argument,
the former using the fact that the positive beads in $\lar{z-1}_{>0}$ will 
move one space to the right.
\end{proof}

While we have stated the above theorem only for $0\leq z\leq t-1$, it may be
extended to arbitrary $z\geq 0$. Since this is not as elegant as the above,
we now state this separately as a corollary.

\begin{coro}\label{Coro_zAsymPrime}
Let $t\geq 2$ and $z=at+b$ be integers and $\la$ a partition such that
$0\leq b\leq t-1$ and $\la\in\mathscr{P}_z$.
Then $\kappa_t(\tcore(\la))\in \mathcal{C}_{b;t}$ and the quotient
$(\lar{0},\dots,\lar{t-1})$ is such that 
for $0\leq r\leq b-1$ with $c_r\geq0$ there exist partitions
$\nur{r}$ with 
\begin{subequations}\label{Eq_folding2}
\begin{equation}
\lar{r}=\nur{r}+((a+1)^{c_r+\rk_{c_r}(\nur{r})})\quad\text{and}\quad
\lar{b-r-1}=(\nur{r})' + ((a+1)^{\rk_{c_r}(\nur{r})}), \label{Eq_bFold}
\end{equation}
and for $b\leq s\leq t-1$ with $c_s\geq0$ there exist partitions
$\xir{s}$ with
\begin{equation}
\lar{s}=\xir{s}+(a^{c_s+\rk_{c_s}(\xir{s})})\quad\text{and}\quad
\lar{t+b-s-1}=(\xir{s})' + (a^{\rk_{c_s}(\xir{s})})\label{Eq_tFold2}.
\end{equation}
\end{subequations}
\end{coro}

This corollary follows simply from the observation that $t$ iterations of the 
``cut and twist'' map used in the previous proof shift all beads at positive
positions one place to the right. 
If $b$ is odd then \eqref{Eq_bFold} says that 
$\lar{(b-1)/2}\in\mathscr{P}_{a+1}$ and if $t+b$ is odd then
\eqref{Eq_tFold2} says that $\lar{(t+b-1)/2}\in\mathscr{P}_a$.
Since we may obtain negative $z$ by 
conjugation, Corollary~\ref{Coro_zAsymPrime} gives a characterisation of
$z$-asymmetric partitions under the Littlewood decomposition.

Our next corollary,  which will prove useful in the statement of our main 
results, characterises when a $t$-core is $z$-asymmetric, and gives the minimal 
$z$-asymmetric partition with a given core.
The first part of this is due to Ayyer and Kumari \cite[Lemma~3.6]{AK22}
in a slightly different form.

\begin{coro}\label{Cor_minimal}
A $t$-core $\mu$ is $z$-asymmetric if and only if $0 \leq z\leq t-2$ and
$\kappa_t(\mu)$ satisfies $c_r=0$ for $0\leq r\leq z-1$.
Moreover, for any sequence $\mathbf{c}\in \mathcal{C}_{z;t}$ the unique 
$z$-asymmetric partition $\la$ with $\kappa_t(\tcore(\la))=\mathbf{c}$
and minimal $\abs{\la}$ 
has quotient $\lar{r}=(1^{c_r})$ for those $r$ with $0\leq r\leq z-1$ and
$c_r>0$.
\end{coro}
\begin{proof}
By Theorem~\ref{Thm_zAsym} a $z$-asymmetric partition $\mu$ must have
$\kappa_t(\mu)\in \mathcal{C}_{z;t}$ and 
$\lar{r}=\varnothing$ for all $0\leq r\leq t-1$.
However, the restrictions \eqref{Eq_zFold} admit the empty partition as a
solution if and only if $c_r=0$.
The second part of the corollary is then immediate.
\end{proof}

This shows that while the $t$-core of a $0$- or $1$-asymmetric partition is 
always itself $0$- or $1$-asymmetric, the same is not necessarily
true for $z$-asymmetric partitions when $z\geq 2$.
Indeed, our running example of $(6,5,5,1)$ is $2$-asymmetric but has 
$t$-core $(1,1)$ which is clearly not $2$-asymmetric.

A key tool we need below is an expression for the Frobenius rank of a 
partition in terms of the Frobenius ranks of its core and quotient. 
This is due to Brunat and Nath, however, we restate it in our terminology and 
provide a short proof.
Some related results about the Frobenius ranks of $-1$-, $0$- and
$1$-asymmetric partitions may be found in \cite[Lemma~3.13]{AK22}.

\begin{lemma}[{\cite[Corollary~3.29]{BN19}}]\label{Lem_rank}
For any partition $\la$ and integer $t\geq 2$,
\[
\rk(\la)=\rk(\tcore(\la))+\sum_{r=0}^{t-1}\rk_{c_r}(\lar{r}).
\]
\end{lemma}
\begin{proof}
Let $\kappa_t(\la)=(c_0,\dots,c_1)$.
As we have already remarked, $\rk(\la)$ is equal to the number of beads
at positive positions in the Maya or $t$-Maya diagram, i.e.,
$\rk(\la)=\sum_{r=0}^{t-1} c_r^+$.
A simple rewriting of this expression gives
\[
\sum_{r=0}^{t-1} c_r^+=\sum_{\substack{r=0\\ c_r>0}}^{t-1} (c_r^+-c_r^-)
+\sum_{\substack{r=0\\ c_r>0}}^{t-1}c_r^-
+\sum_{\substack{r=0\\ c_r\leq 0}}^{t-1}c_r^+.
\]
The first sum on the right is equal to $\rk(\tcore(\la))$ since, after pushing
all beads to the left, this will count the beads remaining at positive 
positions.
If $c_r=0$ then the beads on runner $r$ do not contribute to $\rk(\tcore(\la))$
and so $\rk(\lar{r})=c_r^+$.
Now consider the case $c_r>0$. Counted from the right, the first $c_r$ beads
in this runner are already accounted for by $\rk(\tcore(\la))$.
The quantity $c_r^-=c_r^+-c_r$ then counts the number of remaining beads
at positive positions, which is equal to the Frobenius rank of $\la$ with 
the first $c_r$ rows removed, i.e., to $\rk_{c_r}(\la)$.
By conjugation the same argument works in the case $c_r<0$, completing
the proof.
\end{proof}

Observe that if the $t$-core of $\la$ is empty then 
\[
\rk(\la)=\sum_{r=0}^{t-1} \rk(\lar{r}),
\]
since $\rk_0(\lar{r})=\rk(\lar{r})$.
An example of the computation of the Frobenius rank using the lemma 
is given in Figure~\ref{Fig_foranexample}.
\begin{figure}[htb]
\centering
\begin{tikzpicture}[scale=.4]
\node at (7,-5.5) {$\la$};
\node at (10,-5.5) {$\longmapsto$};
\node at (14,-5.5) {$\xcore{3}(\la)$};
\node at (18,-5.5) {$\lar{0}$};
\node at (20.5,-5.5) {$\lar{1}$};
\node at (23,-5.5) {$\lar{2}$};
\foreach \i [count=\ii] in {8,4,3,3,3,1,1}
\foreach \j in {1,...,\i}{\draw (\j,1-\ii) rectangle (\j+1,-\ii);}
\foreach \i [count=\ii] in {2}
\foreach \j in {1,...,\i}{\draw (12+\j,-1-\ii) rectangle (\j+13,-2-\ii);}
\foreach \i [count=\ii] in {1}
\foreach \j in {1,...,\i}{\draw (16+\j,-1-\ii) rectangle (\j+17,-2-\ii);}
\foreach \i [count=\ii] in {2}
\foreach \j in {1,...,\i}{\draw (18+\j,-1-\ii) rectangle (\j+19,-2-\ii);}
\foreach \i [count=\ii] in {2,2}
\foreach \j in {1,...,\i}{\draw (21+\j,-1-\ii) rectangle (\j+22,-2-\ii);}
\draw[fill=red,opacity=0.7] (1,0) rectangle (2,-1);
\draw[fill=red,opacity=0.7] (2,-1) rectangle (3,-2);
\draw[fill=red,opacity=0.7] (3,-2) rectangle (4,-3);
\draw[fill=red,opacity=0.7] (13,-2) rectangle (14,-3);
\draw[fill=red,opacity=0.7] (17,-2) rectangle (18,-3);
\draw[fill=red,opacity=0.7] (23,-2) rectangle (24,-3);
\end{tikzpicture}
\caption{The Littlewood decomposition of $\la=(8,4,3,3,3,1,1)$ with $t=3$
and $\kappa_3(\la)=(0,1,-1)$.
The marked cells explain the computation of the Frobenius rank:
the left- and right-hand sides both contain three shaded cells since
the first row of $\lar{1}$ and the first column of $\lar{2}$ are ignored.}
\label{Fig_foranexample}
\end{figure}
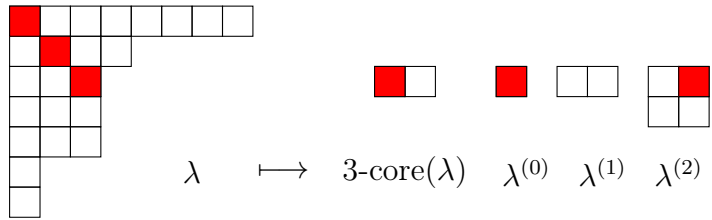

To conclude this section, we give an alternate characterisation of the sign
\eqref{Eq_sgn} in terms of certain permutations.
For a partition $\la$ and an integer $n$ such that $n\geq l(\la)$ we write
$\sigma_t(\la;n)$ for the permutation on $n$ letters which sorts the list
$(\la_1-1,\dots,\la_n-n)$ such that their residues modulo $t$ are increasing, 
and the elements within each residue class are decreasing.
For example if $t=3$, $\la=(6,5,5,1)$ and $n=6$ then the list is
$(5,3,2,-3,-5,-6)$. 
Our permutation is then $\sigma_3(\la;6)=246513$ in one-line notation.
Inversions in this permutation may be read off the $t$-Maya diagram.
They correspond to pairs of beads $(b_1,b_2)$ such that $b_2$ lies 
weakly to the right of and strictly above $b_1$. 
(Note that we only consider the first $n$ beads, read top-to-bottom and
right-to-left.)

In the following lemma we write $\sgn(w)$ for the sign of the permutation
$w$.

\begin{lemma}[{\cite[Lemma~4.5]{Albion23}}]\label{Lem_Sgn}
Let $\la/\mu$ be a $t$-tileable skew shape. Then for any integer 
$n\geq l(\la)$ we have 
\begin{equation}\label{Eq_SgnLem}
\sgn_t(\la/\mu)=\sgn(\sigma_t(\la;n))\sgn(\sigma_t(\mu;n)).
\end{equation}
\end{lemma}
\begin{proof}
Let $\la/\mu$ have ribbon decomposition
\[
\mu=:\nur{0}\subseteq\nur{1}\subseteq\cdots\subseteq\nur{m-1}
\subseteq\nur{m}:=\la,
\]
where $\nur{r}/\nur{r-1}$ is a $t$-ribbon for each $1\leq r\leq m$.
The contribution of the ribbon $\la/\nur{m-1}$ to the sign on the left is
$(-1)^{\mathrm{ht}(\la/\nur{m-1})}$. If the removal of this ribbon corresponds
to moving a bead from position $x$ to $x-t$, then this sign is equal to 
$(-1)^b$ where $b=\abs{\beta(\la)\cap \{x-1,\dots,x-t+1\}}$ counts the number
of beads strictly between $x$ and $x-t$.
By the construction of the permutations $\sigma_t(\la;n)$
and $\sigma_t(\nur{m-1};n)$, we have 
$(-1)^b=\sgn(\sigma_t(\la;n))\sgn(\sigma_t(\nur{m-1};n))$.
In other words, upon removing a single ribbon, both the left- and right-hand
sides of \eqref{Eq_SgnLem} change by the same quantity.
Iterating this completes the proof.
\end{proof}

\section{Generalised universal characters}\label{Sec_Symm}

We now return to symmetric functions.
The first part of this section is devoted to the Verschiebung operator,
defined as the adjoint of the plethysm by a power sum symmetric function.
After briefly surveying its action on various classes of symmetric functions
we state our variants of this action on the universal characters.
This is followed by the main theorems, which compute the image of the 
general symmetric function $\cha_\la$ defined in the introduction.

\subsection{Symmetric functions and plethysm}

Here we give the basic facts relating to symmetric functions;
see \cite[Chapter~1]{Macdonald95} or \cite[Chapter~7]{Stanley99}.
We work in the \emf{algebra of symmetric functions} over $\mathbb{Q}$
and in a countable alphabet $X=(x_1,x_2,x_3,\ldots)$, denoted $\La$.
Important families of symmetric functions we require are the 
\emf{elementary symmetric functions} and the
\emf{complete homogeneous symmetric functions},
defined for integers $k\geq 0$ by
\[
e_k(X):=\sum_{1\leq i_1<\cdots<i_k}x_{i_1}\cdots x_{i_k}
\quad\text{and}\quad
h_k(X):=\sum_{1\leq i_1\leq\cdots\leq i_k}x_{i_1}\cdots x_{i_k},
\]
respectively.
As in the introduction we will drop the alphabet of variables and write 
$e_k$ and $h_k$ for the above.
These are extended to partitions by $h_\la:=h_{\la_1}h_{\la_2}h_{\la_3}\cdots$ 
and analogously for the $e_\la$ and $p_\la$.
The final ``obvious'' basis consists of the \emf{monomial symmetric functions}
\[
m_\la(X):=\sum_{\alpha}x_1^{\alpha_1}x_2^{\alpha_2}x_3^{\alpha_3}\cdots,
\]
where the sum is over all distinct permutations of the partition
$\la=(\la_1,\la_2,\la_3,\ldots)$.
Over $\mathbb{Q}$, all four of the above families form linear bases for
$\La$.

The most important family of symmetric functions are certainly the 
\emf{Schur functions} $s_\la$. The simplest way to define them at the 
generality of skew shapes is by the Jacobi--Trudi determinant
\begin{equation}\label{Eq_JT}
s_{\la/\mu}:=\det_{1\leq i,j\leq l(\la)}(h_{\la_i-\mu_j-i+j}),
\end{equation}
where $h_{-k}:=0$ for $k\geq 1$.
Similarly we have the dual Jacobi--Trudi formula
\[
s_{\la/\mu}=\det_{1\leq i,j\leq \la_1}(e_{\la_i'-\mu_j'-i+j}),
\]
and again $e_{-k}:=0$ for $k\geq 1$.
The symmetric functions $s_\la$ form a basis for $\La$ which is
orthonormal with respect to the Hall inner product \eqref{Eq_Hall-Def}.
Another way to define the skew Schur function is by the adjoint relation
\begin{equation}\label{Eq_SchurAdj}
\langle s_{\la/\mu}, f\rangle = \langle s_\la, s_\mu f\rangle.
\end{equation}
for any $f\in\La$. As already covered above, the 
\emf{Littlewood--Richardson coefficients} $c_{\mu\nu}^\la$ are the structure 
constants of the Schur basis: 
\[
s_\mu s_\nu = \sum_{\la} c_{\mu\nu}^\la s_\la.
\]
Combining this with \eqref{Eq_SchurAdj} in the case $f=s_\nu$ then gives
\[
s_{\la/\mu} = \sum_\nu c_{\mu \nu}^\la s_\nu.
\]
From these equations one sees that $c_{\mu\nu}^\la$ is symmetric in $\mu,\nu$ 
and will vanish unless $\mu,\nu\subseteq\la$ and $\abs{\la}=\abs{\mu}+\abs{\nu}$. These properties extend analogously to the multi-Littlewood--Richardson
coefficients.

We also have the following orthogonality relations among other symmetric 
functions
\begin{equation}\label{Eq_hmpOrth}
\langle h_\la, m_\mu\rangle = \delta_{\la\mu}
\quad\text{and}\quad \langle p_\la,p_\mu\rangle =z_\la\delta_{\la\mu},
\end{equation}
where $\delta_{\la\mu}$ is the usual Kronecker delta and 
$z_\la:=\prod_{i\geq 1}m_i(\la)! i^{m_i(\la)}$.
It is customary to define a homomorphism on symmetric functions
by $\omega h_k=e_k$, which is in fact an involution. 
One may show using the Jacobi--Trudi formulae that 
$\omega s_{\la/\mu}=s_{\la'/\mu'}$. Moreover, $\omega m_\la=f_\la$
where the $f_\la$ are the \emf{forgotten symmetric functions}, and
$\omega$ is an isometry.

\emf{Plethysm} is a composition of symmetric functions first 
introduced by Littlewood \cite{Littlewood36}; see also
\cite[p.~135]{Macdonald95}.
In the introduction we defined the plethysm by a power sum symmetric function
$p_t$, which raises each variable to the power of $t$.
Some properties of this plethysm are
$f\circ p_t = p_t \circ f$ and $p_s \circ p_t = p_{st}$ for $s,t\in\mathbb{N}$.
Moreover if $f$ is homogeneous of degree $n$ then
 \cite[Exercise~7.8]{Stanley99}
\begin{equation}\label{Eq_omegapt}
\omega (f\circ p_t) = (-1)^{n(t-1)}(\omega f)\circ p_t.
\end{equation}
Recall from earlier that the adjoint of this plethysm with respect to the 
Hall scalar product is denoted $\varphi_t$, the $t$-th Verschiebung operator.
We take the adjoint relation as defining this operator; the definition in
\eqref{Eq_premik} given at the very beginning will serve as a special 
case of the following.
The next proposition gives the action of this operator on most of the families
of symmetric functions we have seen so far.
We also provide short proofs of these claims, writing
$\la/t$ as short-hand for the partition $(\la_1/t,\la_2/t,\la_3/t,\ldots)$
when all parts of $\la$ are divisible by $t$.

\begin{prop}\label{Prop_hep}
Let  $t\geq2$ be an integer and $\la$ a partition. 
If $t$ does not divide each part of $\la$ then 
$\varphi_t h_\la=\varphi_t e_\la=\varphi_t p_\la=0$.
If it does, then
\begin{equation}
\varphi_t h_{\la} =h_{\la/t}, \qquad
\varphi_t e_{\la} =(-1)^{\abs{\la}(t-1)/t} e_{\la/t}
\quad\text{and}\quad
\varphi_t p_\la = t^{l(\la)} p_{\la/t}.
\end{equation}
\end{prop}
\begin{proof}
Beginning with the complete homogeneous symmetric function case, it 
is clear from the definition of the $m_\mu$ that $m_\mu\circ p_t= m_{t\mu}$.
Therefore
\[
\langle \varphi_t h_\la, m_\mu\rangle =\langle h_\la,m_\mu \circ p_t\rangle
=\langle h_\la,m_{t\mu}\rangle = \delta_{\la, t\mu},
\]
where the last equality is an application of \eqref{Eq_hmpOrth}.
This implies that $\varphi_t h_\la=0$ if $t$ does not divide each part of
$\la$. 
If it does, then the above is equal to $\delta_{\la/t,\mu}$, which implies
that $\varphi_t h_\la = h_{\la/t}$. 

For the second case, note that since $\omega$ is an isometry
\[
\langle \varphi_t e_\la, f_\mu\rangle 
=\langle e_\la, f_\mu \circ p_t\rangle
=\langle h_\la, \omega(f_\mu\circ p_t)\rangle.
\]
By \eqref{Eq_omegapt} we now have 
$\omega(f_\mu\circ p_t)=(-1)^{\abs{\mu}(t-1)}m_{t\mu}$.
Therefore
\[
\langle \varphi_t e_\la, f_\mu\rangle 
=(-1)^{\abs{\mu}(t-1)}\langle h_\la, m_{t\mu}\rangle
=(-1)^{\abs{\mu}(t-1)}\delta_{\la,t\mu},
\]
again with the aid of \eqref{Eq_hmpOrth}.
Exactly as before this implies that $\varphi_t e_\la=0$ unless all parts of
$\la$ are divisible by $t$. If they are then 
$\varphi_t e_\la=(-1)^{\abs{\la}(t-1)/t}e_{\la/t}$, completing the proof of 
this case.

The power sum case is almost identical.
First we use \eqref{Eq_hmpOrth} to obtain
\[
\langle \varphi_t p_\la,p_\mu\rangle
=\langle p_\la,p_\mu\circ p_t\rangle
=\langle p_\la,p_{t\mu}\rangle
=z_{\la}\delta_{\la,t\mu}.
\]
This tells us that $\varphi_t p_\la$ vanishes unless all parts of $\la$ are
divisible by $t$.
Thus the power sum expansion of $\varphi_t p_\la$ has a single term 
with coefficient $z_{\la}/z_{\la/t}=t^{l(\la)}$.
\end{proof}

The actions of $\varphi_t$ presented in the previous proposition are all
rather simple, and follow the same pattern of dividing all parts of the 
partition by $t$ if possible.
A much richer structure underlies the action of the $t$-th Verschiebung 
operator on the (skew) Schur functions, utilising Littlewood's core and 
quotient construction.

\begin{theorem}\label{Thm_skewSchur}
For any integer $t\geq 2$ and skew shape $\la/\mu$ we have 
$\varphi_t s_{\la/\mu}=0$ unless $\la/\mu$ is $t$-tileable, in which
case
\[
\varphi_t s_{\la/\mu} = \sgn_t(\la/\mu)\prod_{r=0}^{t-1}s_{\lar{r}/\mur{r}},
\]
where the sign is defined in \eqref{Eq_sgn}.
\end{theorem}
For $\mu=\varnothing$ this reduces to Theorem~\ref{Thm_Littlewoodvarphi}
of the introduction.
As alluded to there, the skew case was first worked out by Farahat,
but only when $\mu=\tcore(\la)$ \cite{Farahat58}.
To our knowledge, the first statement of the full skew Schur case appears
in the second edition of Macdonald's book as an example 
\cite[p.~92]{Macdonald95}.
It then makes a further appearance in the work of Lascoux, Leclerc and Thibon
\cite[p.~1049]{LLT97}, which cites Kerber, S\"anger and Wagner
\cite{KSW81}.
However, the latter does not use Schur functions, and rather gives a new 
proof of Farahat's skew generalisation of Theorem~\ref{Thm_LittlewoodMult}
using ``Brettspiele'', which are essentially our Maya diagrams.
In none of these references is the vanishing described in terms of the
tileability of the skew shape $\la/\mu$, with this observation coming
from Evseev, Paget and Wildon \cite[Theorem~3.3]{EPW14} in the context of 
symmetric group characters (where the term $t$-decomposable is used rather 
than our $t$-tileable).
In the precise form above this appears in \cite[Theorem~3.1]{Albion23}.

In an effort to keep this paper for the most part self-contained we now 
provide a proof using Macdonald's approach, which is the same as that of
Farahat.

\begin{proof}[Proof of Theorem~\ref{Thm_skewSchur}]
The first step of the proof is clear: apply $\varphi_t$ to the Jacobi--Trudi
formula \eqref{Eq_JT} to obtain
\begin{equation}\label{Eq_varphiSchurProof}
\varphi_t s_{\la/\mu} = \det_{1\leq i,j\leq n}
\big(\varphi_t h_{\la_i-\mu_j-i+j}\big),
\end{equation}
where $n\geq l(\la)$ is a fixed integer. 
An entry $(i,j)$ in this new determinant is nonzero only if 
$\la_i-i\equiv\mu_j-j\Mod{t}$.
In order to group those entries within the same residue class, permute the
rows and columns according to the permutations $\sigma_t(\la;n)$ and
$\sigma_t(\mu;n)$.
The resulting determinant has a block-diagonal structure.
If $\kappa(\la)=(c_0,\dots,c_{t-1})$,  
$\kappa(\mu)=(d_0,\dots,d_{t-1})$ and $n=at+b$ ($0\leq b\leq t-1$), 
then the $r$th block along the main diagonal will have
dimensions $(c_r+a+[r\geq b]) \times (d_r+a+[r\geq b])$.
Here $[\cdot]$ denotes the \emf{Iverson bracket} which is equal to one if 
the statement $\cdot$ is true and zero otherwise.
These blocks will all be square if and only if $\kappa(\la)=\kappa(\mu)$,
i.e., unless $\tcore(\la)=\tcore(\mu)$, and thus the determinant necessarily
vanishes if this is not the case.
It follows from our definition of the $t$-quotient that after applying the 
Verschiebung operator the indices of the complete homogeneous
symmetric functions in the $r$th block along the diagonal are of the form
$h_{\lar{r}_i-\mur{r}_j-i+j}$ where $1\leq i\leq c_r+a+[r\geq b]$ and 
$1\leq j\leq d_r+a+[r\geq b]$.
Thus we have shown that, if $\tcore(\la)=\tcore(\mu)$, then
\[
\varphi_t s_{\la/\mu}=\sgn(\sigma_t(\la;n))\sgn(\sigma_t(\mu;n))
\prod_{r=0}^{t-1}s_{\lar{r}/\mur{r}}.
\]
This product of skew Schur functions will further vanish unless 
$\mur{r}\subseteq\lar{r}$ for each $0\leq r\leq t-1$.
Putting this together with the previous vanishing we determine that
$\varphi_t s_{\la/\mu}$ is zero unless $\la/\mu$ is $t$-tileable, in which 
case it is given by the above 
product. The sign is then equal to $\sgn_t(\la/\mu)$ by
Lemma~\ref{Lem_Sgn}.
\end{proof}

\subsection{Generalised universal characters}

For a finite set of $n$ variables the Schur polynomial $s_\la(x_1,\dots,x_n)$ 
is the character of the irreducible polynomial representation 
of $\mathrm{GL}_n$ indexed by $\la$.
The classical groups $\mathrm{O}_{2n}$, $\mathrm{Sp}_{2n}$
and $\mathrm{SO}_{2n+1}$ also carry irreducible representations 
indexed by partitions.
The characters of these representations are rather Laurent polynomials
symmetric under permutation and inversion of the $n$ variables.
Using the Jacobi--Trudi formulae for these characters,
originally due to Weyl, they may still be expressed as determinants in 
the complete homogeneous symmetric functions of the form
$h_r(x_1,1/x_1,\dots,x_n,1/x_n)$
\cite[Theorems~7.8.E \& 7.9.A]{Weyl39}.
Rather than working with these characters we will use the 
\emf{universal characters}, as defined by Koike and Terada \cite{Koike97,KT87}.
These are symmetric function lifts of the ordinary characters given 
by `forgetting' the variables in Weyl's Jacobi--Trudi formulae:
\begin{subequations}\label{Eq_UChJT}
\begin{align}
\sp_\la &:= \frac{1}{2}\det_{1\leq i,j\leq k}(h_{\la_i-i+j}+h_{\la_i-i-j+2})\\
\o_\la &:= \det_{1\leq i,j\leq k}(h_{\la_i-i+j}-h_{\la_i-i-j}) \\
\so_\la &:= \det_{1\leq i,j\leq k}(h_{\la_i-i+j}+ h_{\la_i-i-j+1}),
\end{align}
\end{subequations}
where $k$ is an integer such that $l(\la)\leq k$.
We also have the dual forms
\begin{subequations}\label{Eq_UCeJT}
\begin{align}
\sp_\la &= \det_{1\leq i,j\leq\ell}(e_{\la_i'-i+j}-e_{\la_i'-i-j}) \\
\o_\la &= \frac{1}{2}\det_{1\leq i,j\leq\ell}(e_{\la_i'-i+j}+e_{\la_i'-i-j+2})\\
\so_\la &= \det_{1\leq i,j\leq\ell}(e_{\la_i'-i+j}+ e_{\la_i'-i-j+1}),
\label{Eq_UCeJTso}
\end{align}
\end{subequations}
where here $\ell$ is an integer such that $\la_1\leq \ell$.
Comparing \eqref{Eq_UChJT} and \eqref{Eq_UCeJT} it is clear that
$\omega \o_\la=\sp_{\la'}$ and $\omega \so_\la=\so_{\la'}$.

Let $\La_n^{\mathrm{BC}}$ denote the ring of Laurent polynomials in 
$x_1,\dots,x_n$ which are symmetric under permutation and inversion of the
variables.
Define for integers $n\geq 1$ the restriction maps 
$\pi_n:\La\longrightarrow\La_n^{\mathrm{BC}}$ by 
$\pi_n(e_r)=e_r(x_1,1/x_1,\dots,x_n,1/x_n)$.
If $r>2n$ then $\pi_n(e_r)=0$, and moreover,
$\pi_n(e_r-e_{2n-r})=0$ for each $0\leq r\leq n$.
For a partition $\la$ with $l(\la)\leq n$ the images of the universal 
characters under $\pi_n$ are the actual characters of their respective groups
indexed by $\la$.
If $l(\la)>n$ then these specialisations either vanish or, up to a sign,
produce an irreducible character of the same group associated to a different 
partition which is determined by the so-called ``modification rules''; see 
\cite{King71} and \cite[\S2]{KT87}.
We also have the modified map $\pi_n'$ which acts by
$\pi_n'(e_r)=e_r(x_1,1/x_1,\dots,x_n,1/x_n,1)$ and satisfies
$\pi_n(\so_{\la})=\pi_n'(\o_\la)$ for $\la$ with $l(\la)\leq n$.

In the introduction we already met the character $\so_\la$
in \eqref{Eq_so-def} and saw that
it could be expanded as a signed sum over skew Schur functions where the inner
shape is a self-conjugate ($0$-asymmetric) partition.
In fact, all three of the characters \eqref{Eq_UChJT} admit such expressions:
\begin{subequations}\label{Eq_UCskew}
\begin{align}
\sp_\la&=\sum_{\mu\in\mathscr{P}_{-1}}(-1)^{\abs{\mu}/2}s_{\la/\mu},
\label{Eq_UCskewsp}\\
\o_\la&=\sum_{\mu\in\mathscr{P}_1}(-1)^{\abs{\mu}/2}s_{\la/\mu},
\label{Eq_UCskewo} \\
\so_\la&=\sum_{\mu\in\mathscr{P}_0}(-1)^{(\abs{\mu}-\rk(\mu))/2}s_{\la/\mu}.
\label{Eq_UCskewso}
\end{align}
\end{subequations}
The Schur functions themselves may be simply expanded in terms of these
universal characters:
\begin{equation}\label{Eq_CIR}
s_\la=
\sum_\mu\bigg(\sum_{\substack{\nu\\\text{$\nu$ even}}}c_{\mu\nu}^\la\bigg)\o_\mu
=
\sum_\mu\bigg(\sum_{\substack{\nu\\\text{$\nu'$ even}}}c_{\mu\nu}^\la\bigg)\sp_\mu
=\sum_\mu\bigg(\sum_{\nu}(-1)^{\abs{\nu}}c_{\mu\nu}^\la\bigg)\so_\la,
\end{equation}
where here we write ``$\nu$ even'' meaning $\nu$ has only even parts.
This last set of equalities are precisely the ``Character Interrelation 
Theorem'' of Koike and Terada; see \cite[Theorem~2.3.1]{KT87}
and \cite[Theorem~7.2]{Koike97}.

While these three are the most well-known universal characters, we need two
more. The first of these is the universal character associated with the 
negative part of the odd orthogonal group
\begin{equation}
\so_\la^-:=
\det_{1\leq i,j\leq k}\big(h_{\la_i-i+j}-h_{\la_i-i-j+1}\big)
=\sum_{\mu\in\mathscr{P}_0}(-1)^{(\abs{\mu}+\rk(\mu))/2}s_{\la/\mu}.
\end{equation}
There is also an $e$-Jacobi--Trudi formula where the sum is replaced by
a difference in \eqref{Eq_UCeJTso}.
Writing $-X:=(-x_1,-x_2,x_3,\dots)$ we further have
$\so_\la^-(X)=(-1)^{\abs{\la}}\so_\la^+(-X)$,
where, in order to avoid confusion, from now on we write $\so^+_\la$ in place
of $\so_\la$.

The next universal character we need is that of an irreducible 
rational representation of $\mathrm{GL}_n$;
see \cite{Koike89,Stembridge87}.
(The universal characters of the polynomial representations are the Schur
functions.)
These representations are indexed by weakly decreasing sequences of integers 
with length exactly $n$, or, alternatively, pairs of partitions
$\la,\mu$ such that $l(\la)+l(\mu)\leq n$.
Given such a pair we let the $i$-th component of the associated
$\mathrm{GL}_n$ weight $[\la,\mu]_n$ be given by
$\la_i-\mu_{n-i+1}$.
Recall that the Schur polynomial $s_\la(x_1,\dots,x_n)$ may be extended to
weakly decreasing sequences of integers of length $n$ by the relation
\begin{equation}
s_{(\la_1+1,\dots,\la_n+1)}(x_1,\dots,x_n)=(x_1\cdots x_n)
s_{(\la_1,\dots,\la_n)}(x_1,\dots,x_n).
\end{equation}
The character of an irreducible rational representation of 
$\mathrm{GL}_n$ is then simply 
$s_{[\la,\mu]_n}(x_1,\dots,x_n)$.
Littlewood gave the following expansion in terms of skew Schur polynomials
\cite{Littlewood44}:
\begin{equation}
s_{[\la,\mu]_n}(x_1,\dots,x_n)
=\sum_{\nu}(-1)^{\abs{\nu}}
s_{\la/\nu}(x_1,\dots,x_n)s_{\mu/\nu'}(1/x_1,\dots,1/x_n).
\end{equation}
Koike used this expression to define a universal character which depends on
two independent alphabets $X=(x_1,x_2,x_3,\dots)$ and 
$Y=(y_1,y_2,y_3,\dots)$ as
\cite{Koike89}\footnote{Our $\rs_{\la,\mu}$ stands for
``rational Schur function''.}
\begin{equation}\label{Eq_rsDef}
\rs_{\la,\mu}(X;Y):=
\sum_{\nu}(-1)^{\abs{\nu}}s_{\la/\nu}(X)s_{\mu/\nu'}(Y),
\end{equation}
which is an element of $\La_X\otimes \La_Y$.
Define the restriction map $\tilde\pi_n:\La_X\otimes\La_Y\longrightarrow
\La_n^\pm$, the space of symmetric Laurent polynomials in $x_1,\dots,x_n$,
by
\begin{equation}\label{Eq_tildepi}
\tilde\pi_n(\rs_{\la,\mu}(X;Y))=
\rs_{\la,\mu}(x_1,\dots,x_n; 1/x_1,\dots,1/x_n)=s_{[\la,\mu]}(x_1,\dots,x_n),
\end{equation}
for $l(\la)+l(\mu)\leq n$. Again, if this final condition is violated then
there are modification rules which allow for the specialisation to be 
expressed as the character of a different rational representation.
This object also has Jacobi--Trudi-type expressions.
In terms of the complete homogeneous symmetric functions the first of these is
\begin{equation}\label{Eq_rsh}
\rs_{\la,\mu}(X;Y)
=\det
\begin{pmatrix}
(h_{\la_i-i+j}(X))_{1\leq i,j\leq k} & 
(h_{\la_i-i-j+1}(X))_{\substack{1\leq i\leq k \\ 1\leq j\leq \ell}} \\
(h_{\mu_i-i-j+1}(Y))_{\substack{1\leq i\leq \ell \\ 1\leq j\leq k}} &
(h_{\mu_i-i+j}(Y))_{1\leq i,j\leq\ell}
\end{pmatrix},
\end{equation}
where $k\geq l(\la)$ and $\ell\geq l(\mu)$.
Again we have the dual form
\[
\rs_{\la,\mu}(X;Y)
=\det
\begin{pmatrix}
(e_{\la_i'-i+j}(X))_{1\leq i,j\leq k} & 
(e_{\la_i'-i-j+1}(X))_{\substack{1\leq i\leq k \\ 1\leq j\leq \ell}} \\
(e_{\mu_i'-i-j+1}(Y))_{\substack{1\leq i\leq \ell \\ 1\leq j\leq k}} &
(e_{\mu_i'-i+j}(Y))_{1\leq i,j\leq\ell}
\end{pmatrix},
\]
where now $k\geq \la_1$ and $\ell\geq\mu_1$.
Exactly how these determinantal representations of $\rs_{\la,\mu}(X;Y)$ and
the skew Schur expansion are related will be explained below.
In what follows we will predominantly use this symmetric function 
for $X=Y$, in which case we suppress the alphabet and simply write
$\rs_{\la,\mu}=\rs_{\la,\mu}(X;X)$.
We have already used this in Theorem~\ref{Thm_so} of the introduction.

In analysing Goulden's combinatorial proof of the Jacobi--Trudi formula
\cite{Goulden85},
Bressoud and Wei \cite{BW92} discovered a uniform extension of \eqref{Eq_UCskew}
involving an integer $z\geq -1$ which reproduces the above for
$z=-1,1,0$ respectively.
This was generalised further by Hamel and King to an expression 
valid for all $z\in\mathbb{Z}$ and including an additional parameter $q$
\cite{HK11a,HK11b}.
Then the main result of Hamel and King is
\begin{subequations}\label{Eq_chiz}
\begin{align}
\cha_{\la}(z;q)&:=\det_{1\leq i,j\leq k}
\big(h_{\la_i-i+j}+[j>-z]qh_{\la_i-i-j+1-z}\big) \label{Eq_chizDet}\\[2mm]
&\hphantom{:}=\sum_{\mu\in\mathscr{P}_z}
(-1)^{(\abs{\mu}-\rk(\mu)(z+1))/2}q^{\rk(\mu)}s_{\la/\mu},\label{Eq_chizSkew}
\end{align}
\end{subequations}
where $k$ is an integer such that $k\geq l(\la)$ and we have used the 
Iverson bracket from the proof of Theorem~\ref{Thm_skewSchur}.
Their paper \cite{HK11a} provides a proof of the identity 
$\eqref{Eq_chizDet}=\eqref{Eq_chizSkew}$ using the Laplace expansion of the
determinant, whereas in \cite{HK11b} a combinatorial proof is provided
using the Lindstr\"om--Gessel--Viennot lemma \cite{GV89}.
The general symmetric function $\cha(z;q)$ also reduces to the three 
classical cases, but in a slightly different manner to the determinant
of Bressoud and Wei:
\[
\sp_{\la}=\cha_\la(-1;1), \quad
\o_{\la}=\cha_\la(1;-1),\quad \text{and}\quad
\so_{\la}^\pm=\cha_\la(0;\pm1).
\]
The expansion in terms of skew Schur functions immediately implies the 
following duality with respect to the involution $\omega$:
\begin{equation}\label{Eq_chiDual}
\omega\cha_{\la}(z;q)=\cha_{\la'}(-z;(-1)^zq).
\end{equation}
This extends $\omega\o_\la=\sp_{\la'}$. 

The symmetric function $\cha_\la(z;q)$ is the subject of the first main result
of Hamel and King in \cite{HK11a,HK11b}. They also introduce a generalisation
of the determinantal form of $\rs_{\la,\mu}(X;Y)$ \eqref{Eq_rsh} in a similar
vein, involving two parameters $u,v$ and a pair of (possibly negative) 
integers $a,b$.
We express this as
\begin{align}\label{Eq_rsHKab}
&\rs_{\la,\mu}(X;Y;a,b;u,v)\\
&\quad:=\det
\begin{pmatrix}
\big(h_{\la_i-i+j}(X)\big)_{1\leq i,j\leq k} &
\big([j>-a]uh_{\la_i-i-j-a+1}(X)
\big)_{\substack{1\leq i\leq k\\ 1\leq j\leq\ell}} \\
\big([j>-b]vh_{\mu_i-i-j-b+1}(Y)\big)_{\substack{1\leq i\leq\ell\\1\leq j\leq k}}
& \big(h_{\mu_i-i+j}\big)_{1\leq i,j\leq \ell}.\notag
\end{pmatrix},
\end{align}
where as usual $k\geq l(\la)$ and $\ell\geq l(\mu)$ are integers.
For $(a,b,u,v)=(0,0,1,1)$ we recover Koike's rational universal character.
Observe that the structure of this determinant, including Iverson brackets, is 
clearly similar to that of $\cha_\la(z;q)$.
Since the determinant is quite complicated, let us give an example
for $(\la,\mu,a,b,k,\ell)=((3,2),(4,2,1,1),-1,2,2,4)$:
\begin{equation}\label{Eq_exampleDet}
\det
\begin{pmatrix}
h_3(X) & h_2(X) & 0 & uh_2(X) & uh_1(X) & u \\
h_1(X) & h_2(X) & 0 & u & 0 & 0 \\
vh_1(Y) & v & h_4(Y) & h_5(Y) & h_6(Y) & h_7(Y) \\
0 & 0 & h_1(Y) & h_2(Y) & h_3(Y) & h_4(Y) \\
0 & 0 & 0 & 1 & h_1(Y) & h_2(Y) \\
0 & 0 & 0 & 0 & 1 & h_1(Y)
\end{pmatrix}.
\end{equation}
Using both algebraic and lattice path techniques, Hamel and King show that this
more general symmetric function expands nicely in terms of skew Schur functions
\cite[Theorem~2]{HK11a}
\[
\rs_{\la,\mu}(X;Y;a,b;u,v)
=\sum_{\nu}(-1)^{\abs{\nu}}(uv)^{\rk(\nu)}
s_{\la/(\nu+a^{\rk(\nu)})}(X)s_{\mu/(\nu'+b^{\rk(\nu)})}(Y),
\]
where the sum is over all partitions $\nu=(a_1,\dots,a_k\mid b_1,\dots,b_k)$
of arbitrary Frobenius rank such that $a_r\geq \max\{0,-a\}$ and
$b_r\geq \max\{0,-b\}$ for $1\leq r\leq k=\rk(\nu)$.
For example, in computing $\rs_{(3,2),(4,2,1,1)}(X;Y;-1,2;u,v)$ from
\eqref{Eq_exampleDet} the term $\nu=(1)$ is excluded from the sum since
$(1)=(0\mid0)$ in Frobenius notation.
Intuitively, this ensures that the Frobenius rank of the partition
$\nu+(a^{\rk(\nu)})$ is never less than the Frobenius rank of $\nu$.

A variant of the Koike and Hamel--King determinants 
involving an additional positive integer $c$ occurs naturally in our factorisation 
results for universal characters below. 
Here we write $[k,\ell]:=(k+1,\dots,\ell)$, which we treat as empty for
$k\geq \ell$.
The modified Hamel--King determinant is defined by the identity
\begin{align*}
&u^{c}(-1)^{kc+\binom{c}{2}}\rs_{\la,\mu}(X;Y;a,b;c;u,v)\\
&=
\!\det\begin{pmatrix}
(h_{\la_i-i+j}(X))_{\substack{i\in[0,k]\\j\in[c,k]}} & \!\!\!\!\!
\big([j>-a-c]uh_{\la_i-i-j-a+1}(X)\big)_{\substack{i\in[0,\max\{k,c\}]\\j\in [-c,\ell]}} \\
\big([j>-b]vh_{\mu_i-i-j+1-b}(Y)\big)_{\substack{i\in[0,\ell]\\ j\in[c,k]}} & \!\!\!\!\!
(h_{\mu_i-i+j}(Y))_{\substack{i\in[0,\ell]\\j\in[-c,\ell]}}
\end{pmatrix}
\end{align*}
where $k\geq l(\la)$ and $\ell\geq l(\mu)$.
For $c=0$ this reduces to the Hamel--King determinant \eqref{Eq_rsHKab}.
While not entirely clear from the definition, this determinant 
does not depend on $k$ or $\ell$ as long the length conditions hold.
In the case that $c\geq k$ the two sub-matrices on the left do not appear due to
having no valid column indices.
Like Koike's character, this also has an expansion in terms of skew Schur 
functions. Recall from Theorem~\ref{Thm_zAsym} 
that $\rk_c(\la)$ denotes the Frobenius rank of the
partition obtained by removing the first $c$ rows of $\la$.

\begin{theorem}
For partitions $\la,\mu$, and integers $a,b,c$ such that
$c\geq 0$ we have
\[
\rs_{\la,\mu}(X;Y;a,b;c;u,v)
=\sum_{\nu}(-1)^{\abs{\nu}}(uv)^{\rk_c(\nu)}
s_{\la/(\nu+(a^{c+\rk_c(\nu)}))}(X)s_{\mu/(\nu'+(b^{\rk_c(\nu)}))}(Y),
\]
where the sum is over all partitions $\nu$ for which 
$\rk_c(\nu)=\rk_c(\nu+(a^{c+\rk_c(\nu)}))$.
\end{theorem}
\begin{proof}
The technique is the same as in \cite[p.~68]{Koike89} and \cite[p.~553]{HK11a}.
Without loss of generality assume that $k\geq c$.
We apply the Laplace expansion to the determinantal form of
$\rs_{\la,\mu}(X;Y;a,b;c;u,v)$ according to the given block structure,
choosing the first $k$ rows to be fixed.
We index the sum by permutations 
$w\in \mathfrak{S}_{k+\ell}/(\mathfrak{S}_k\times\mathfrak{S}_\ell)$
acting on the set $\{c+1,\dots,k\}\cup\{c,\dots,1-\ell\}$.
(In other words, the first $k$ columns are labelled $c+1,\dots, k$ and the final
$\ell$ columns $c,\dots, 1-\ell$.)
Define the sets $K_w:=\{w(j): 1\leq j\leq k\}$ and
$L_w:=\{w(j):1-\ell\leq j\leq 0\}$.
Then the Laplace expansion of $\rs_{\la,\mu}(X;Y;a,b;c;u,v)$ may be expressed as
\begin{align*}
(-1)^{kc+\binom{c}{2}}
\sum_{w\in\mathfrak{S}_{k+\ell}/(\mathfrak{S}_k\times \mathfrak{S}_\ell)}
\sgn(w)u^{r-c} v^s
\det_{\substack{1\leq i\leq k\\j\in K_w}}(\alpha_jh_{\la_i-i+p_j}(X))
\det_{\substack{1\leq i\leq \ell\\j\in L_w}}(\beta_jh_{\mu_i-i-q_j+1}(Y)),
\end{align*}
where we set
\begin{equation*}
\alpha_j:=\begin{cases} 
1 & \text{if $c+1\leq j\leq k$}, \\
0 & \text{if $c-a+1\leq j\leq c$}, \\
1 & \text{if $1-\ell\leq j\leq c-a$},
\end{cases}
\quad\text{and}\quad
\beta_j
:=\begin{cases}
0 & \text{if $c+1\leq j\leq c-b-1$}, \\
1 & \text{if $c-b\leq j\leq k$}, \\
1 & \text{if $1-\ell\leq j\leq c$},
\end{cases}
\end{equation*}
which encode the Iverson brackets from the full determinant,
\begin{equation}\label{Eq_detProofpq}
p_j:=\begin{cases}
j & \text{if $c+1\leq j\leq k$}, \\
j-a & \text{if $1-\ell\leq j\leq c$},
\end{cases}
\quad\text{and}\quad
q_j:=\begin{cases}
j+b &\text{if $c+1\leq j\leq k$},\\
j &\text{if $1-\ell\leq j\leq c$}. 
\end{cases}
\end{equation}
We also have the quantities
$r=\{1\leq j \leq k:1-\ell\leq w(j)\leq c\}$ and
$s=\{1-\ell\leq j\leq c : c+1\leq w(j)\leq k\}$.

As a next step we reverse the order of the columns labelled $c,\dots,1$ and
then move them to the left, which cancels the overall sign from the 
determinant defining our symmetric function.
Treating the $w$ as permutations of the set 
$\{1,\dots,k\}\cup\{0,\dots,1-\ell\}$ 
we then choose coset representatives such that 
$w(1)<\cdots<w(k)$ and $w(-1)>\cdots>w(1-\ell)$ ordered 
canonically as integers.
For example, in two-line notation with $(k,\ell)=(3,2)$ one such coset 
representative is
\[
\begin{pmatrix}
1 & 2 & 3 & 0 & -1 \\
-1 & 2 & 3 & 1 & 0 
\end{pmatrix}.
\]
The coset representatives of 
$\mathfrak{S}_{k+\ell}/(\mathfrak{S}_k\times\mathfrak{S}_\ell)$ and
partitions $\nu\subseteq(k^\ell)$ are in bijection; 
see \cite[p.~3]{Macdonald95} or \cite[p.~553]{HK11a}.
The assignment $w(i)=i-\nu_i$ if $1\leq i\leq k$ and
$w(i)=\nu_{1-i}'+i$ for $1-\ell\leq i\leq 0$ gives the corresponding partition.
In the above example we obtain $\nu=(2)$, and the sign of the 
permutation will be equal to $=(-1)^{\abs{\nu}}$.
Moreover, $r-c=s=\rk_c(\nu)$.
To complete the proof we need only observe that the definitions of
$p_j$ and $q_j$ from \eqref{Eq_detProofpq} imply that the terms in the sum 
which are nonzero come from partitions $\nu$ for which
$\nu+(a^{c+\rk_c(\nu)})\subseteq\la$ and
$\nu'+(b^{\rk_c(\nu)})\subseteq\mu$.
We must also have that $\rk_c(\nu)=\rk_c(\nu+(a^{c+\rk_c(\nu)}))$.
This completes the proof.
\end{proof}

\subsection{Restriction rules}
Later on we will require some general restriction rules due to Koike and Terada.
The purpose of this subsection is to collect these results.
The first of these rules gives the restriction of an irreducible $\mathrm{GL}_n$
character to any subgroup of the form $\mathrm{GL}_k\times \mathrm{GL}_{n-k}$
where $0\leq k\leq n$.
This uses the restriction homomorphism $\tilde\pi_n$ defined in
\eqref{Eq_tildepi}.

\begin{theorem}[{\cite[Proposition~2.6]{Koike89}}]\label{Thm_GLGLGL}
Let $\la,\mu$ be partitions such that $l(\la)+l(\mu)\leq n$.
Then for any integer $k$ such that $0\leq k\leq n$,
\begin{equation}\label{Eq_GLGLGL}
\tilde\pi_n(\rs_{\la,\mu}(X;Y))
=\sum_{\nu,\xi,\rho,\tau}\bigg(\sum_\eta c_{\nu,\rho,\eta}^\la c_{\xi,\tau,\eta}^\mu\bigg)
\tilde\pi_k(\rs_{\nu,\xi}(X;Y))\tilde\pi_{n-k}(\rs_{\rho,\tau}(X;Y)).
\end{equation}
\end{theorem}

The next result gives the restriction of $\mathrm{SO}_{2n+1}$ to a 
maximal parabolic subgroup $\mathrm{GL}_k\times \mathrm{SO}_{2(n-k)+1}$.
We write $\pi_{n;Z}$ and $\tilde\pi_{n;X,Y}$ for the restriction maps acting
on the labelled sets of variables.

\begin{theorem}[{\cite[Theorem~2.1]{KT90}}]\label{Thm_SOGLSO}
For any partition $\la$ and integers $k,n$ such that $0\leq k\leq n$, then
\begin{equation}\label{Eq_SOGLSO}
\tilde\pi_{k;X,Y}\pi_{n-k;Z}(\so_\la^+(X,Y,Z))
=\sum_{\mu,\nu,\xi} \bigg(\sum_\eta c^\la_{\mu,\nu,\xi,\eta,\eta}\bigg)
\tilde\pi_{k}(\rs_{\mu,\nu}(X;Y))\pi_{n-k}(\so^+_\xi).
\end{equation}
\end{theorem}
Care needs to be taken in considering the case $k=n$, in which case 
$\pi_0$ will extract the constant term of $\so_\la^+$. This may be computed
by \eqref{Eq_UCskewso} and is
\[
\pi_0(\so^+_\la)=\begin{cases}
(-1)^{(\abs{\la}-\rk(\la))/2} & \text{if $\la\in\mathscr{P}_0$}, \\
0 & \text{otherwise}.
\end{cases}
\]

Finally, we also have an expression for the restriction of 
$\mathrm{SO}_{2n+1}$ to $\mathrm{GL}_n$, which is different to the $k=0$ case
above.
\begin{theorem}[{\cite[Theorem~A.1]{KT90}}]\label{Thm_SOGL}
For a partition $\la$ of length at most $n$ we have that
\begin{equation}\label{Eq_SOGL}
\tilde\pi_n(\so_{\la}^+(X,Y))
=\sum_{\mu,\nu}\bigg(\sum_{\eta}c_{\mu,\nu,\eta}^\la\bigg)
\tilde\pi_n(\rs_{\mu,\nu}(X;Y)).
\end{equation}
\end{theorem}
The difference between the $k=n$ case of \eqref{Eq_SOGLSO} and \eqref{Eq_SOGL} 
is that for $n\geq 2l(\la)$ the latter will contain only positive terms, and 
there is no need for modification rules in the computation of the sum.  
In general, the restriction maps in Theorem~\ref{Thm_SOGLSO} can be removed
if $k\geq 2l(\la)$ and $n-k\geq l(\la)$, but of course this excludes the
case $n=k$; see \cite[Corollary~2.3]{KT90}.

\section{Factorisations of universal characters}\label{Sec_proofs}
We now turn to the factorisation of the universal characters under the operator
$\varphi_t$.
As a first step we state the universal character lifts of the factorisation 
results of Lecouvey from \cite{Lecouvey09B}, amounting to 
the computation of $\varphi_t \o_\la$, $\varphi_t\sp_\la$ and
$\varphi_t \so_\la$ in our notation.
These also include the results of Ayyer and Kumari \cite{AK22} as special cases.
As seen in the previous section, these universal characters have a 
uniform generalisation in Hamel and King's symmetric function $\cha_\la(z;q)$.
Our main result is the computation of $\varphi_t \cha_\la(z;q)$ for all
$z\in\mathbb{Z}$, amounting to a large generalisation of the results of
Lecouvey and Ayyer and Kumari.

\subsection{Factorisations of classical characters}\label{Sec_subClass}
In \cite{Lecouvey09B}, Lecouvey sought generalisations of the LLT polynomials 
beyond type A.
To achieve this goal he needed analogues of the action of the $t$-th
Verschiebung operator on the Schur polynomials,
Theorem~\ref{Thm_skewSchur}, for the symplectic and orthogonal
characters.
These results, stated in \cite[Section~3.2]{Lecouvey09B},
give conditions on the vanishing of the characters under $\varphi_t$ for all
$t$. In addition, with the restriction that $t$ must be odd in the 
symplectic and odd orthogonal cases, he expresses the result using branching
coefficients involving a subgroup of Levi type.
In subsequent work \cite{Lecouvey09A}, he used these factorisations to
give expressions for the plethysm $\so_\la^+\circ p_t$ and its cousins by 
passing from the characters to the universal characters.
For $t=2$ some preliminary work towards the computation of these twisted 
characters was done by Mizukawa \cite{Mizukawa02}.

Independently of the results of Lecouvey, Ayyer and Kumari also proved 
expressions for the action of the $t$-th Verschiebung operator on the 
(non-universal) symplectic and orthogonal characters, but rather phrased in
terms of ``twisting'' by a root of unity \cite{AK22}.
There are, however, key differences between their results and those of 
Lecouvey.
Their twisted character identities, when they are nonzero, factor as products of
other characters.
Moreover, they give nicer conditions for when the characters are nonzero.
Namely, they show that $\o_\la, \sp_\la$ and $\so_\la^+$ vanish under
$\varphi_t$ if and only if the $\tcore$ of $\la$ is 
$1$-, $-1$- or $0$-asymmetric respectively.\footnote{Note that $1$-asymmetric partitions have several names including
\emf{threshold partitions} or \emf{doubled distinct partitions}.}
Lifts of the results of Lecouvey, Ayyer and Kumari to the universal 
characters were given in \cite{Albion23}.
The proofs there, like Lecouvey's, are based on the Jacobi--Trudi formulae
for the classical groups.
Note that in \cite{AK22}, due to twisting by a primitive $t$-th root of unity, 
the characters are associated with a Lie group of rank $nt$.
In \cite{Lecouvey09B}, no such restriction on the rank $n$ is assumed, only 
that the length of the partition indexing the character is at most $n$.
Indeed, depending on the remainder of $n$ modulo $t$ the structure of the 
factorisation of the classical characters will change.
We will discuss this more below in the case of $\mathrm{SO}_{2n+1}$, and show
that the construction of Lecouvey may be phrased in terms of the classical 
Littlewood decomposition.
Note that in Schur case it is clear from Theorem~\ref{Thm_skewSchur} that 
cyclic permutations of the quotient do not change the result, and so 
no such distinction must be made in that case.

Recall that in the skew Schur function case, when it is nonzero, the sign of 
$\varphi_t s_{\la/\mu}$ may be expressed elegantly in terms of the
$t$-ribbon tiling of the skew shape $\la/\mu$.
In all previous work the signs obtained by applying the operator $\varphi_t$
to the ordinary and universal characters were not expressed in such a 
combinatorial manner,
rather as the sign of a permutation multiplied by some further factors
to account for matrix operations occurring in the proof.
In the present work we are able to improve on this by giving explicit 
expressions for the signs based on tilings of skew shapes and statistics on
the indexing partitions, as already exemplified in Theorem~\ref{Thm_so} of
the introduction.

Let us now state the two missing cases, beginning with the even orthogonal
universal character.
In these results we again write $\tilde\la$ as shorthand for the $t$-core of
$\la$. 
We also set $\rs_{\la,\mu}:=\rs_{\la,\mu}(X;X;0,0;0;1,1)$.

\begin{theorem}\label{Thm_o}
For all $t\geq 2$ and a partition $\la$ we have that $\varphi_t\o_\la$ vanishes
unless $\tcore(\la)\in\mathscr{P}_1$, in which case
\[
\varphi_t \o_\la
=(-1)^{\abs{\tilde\la}/2}\sgn_t(\la/\tilde\la)
\o_{\lar{0}}\prod_{r=1}^{\lfloor(t-2)/2\rfloor}
\rs_{\lar{r},\lar{t-r}}
\times\begin{cases}
\so_{\lar{t/2}}^- & \text{$t$ even}, \\
1 & \text{$t$ odd}.
\end{cases}
\]
\end{theorem}

If $l(\la)\leq n$ then restricting to $\La_n^{\mathrm{BC}}$ 
recovers \cite[Theorem~2.15]{AK22}.
Secondly, we have the symplectic case.

\begin{theorem}\label{Thm_sp}
For all $t\geq 2$ and a partition $\la$ we have that $\varphi_t\sp_\la$ vanishes
unless $\tilde\la\in\mathscr{P}_{-1}$, in which case
\[
\varphi_t \sp_\la
=(-1)^{(\abs{\tilde\la}+\rk(\tilde{\la}))/2}\sgn_t(\la/\tilde\la)
\sp_{\lar{t-1}}\prod_{r=0}^{\lfloor(t-3)/2\rfloor}
\rs_{\lar{r},\lar{t-r-2}}
\times\begin{cases}
\so_{\lar{(t-2)/2}} & \text{$t$ even}, \\
1 & \text{$t$ odd}.
\end{cases}
\]
\end{theorem}

As for the previous theorem we may recover \cite[Theorem~2.11]{AK22}.
The odd orthogonal case is given in Theorem~\ref{Thm_so} above and generalises
\cite[Theorem~2.17]{AK22}.

The aforementioned three theorems appeared in
\cite[Theorems~3.2--3.4]{Albion23} with the same signs as in \cite{AK22}.
The expressions for the signs we present here are not only of a more 
combinatorial flavour, but also easier to compute.
Another upshot of these expressions is that they show that the algorithms for
computing the action of $\varphi_t$ on the classical characters in Lecouvey's
work \cite{Lecouvey09B} can be phrased entirely in terms of the Littlewood
decomposition of the underlying partition. 

\subsection{A uniform $(z,q)$-analogue}
The universal characters and the Schur functions are all contained in the
general symmetric function $\cha_\la(z;q)$ of Hamel and King.
Thus, a natural question is whether the operator $\varphi_t$ acts as nicely
on this symmetric function as it does for its special cases.
Our main result is the affirmative answer to this question for all integers
$z$ and including the parameter $q$.

Recall from Corollary~\ref{Cor_minimal} that $\mu_{\mathbf{c}}$ denotes the
minimal $z$-asymmetric partition with $\kappa_t(\mu_{\mathbf{c}})=\mathbf{c}$.
If $z<0$ then the conditions in that corollary need to be conjugated.
From here on out we write $\rs_{\la,\mu}(a;c;q):=\rs_{\la,\mu}(X;X;a,a;c;q,q)$
and extend this to negative $c$ by
$\rs_{\la,\mu}(a;-c;q):=\rs_{\mu,\la}(a;c;q)$.

\begin{theorem}\label{Thm_chiz}
Let $a,b,t,z$ be integers such that $t\geq 2$ and $z=at+b$ where
$0\leq b\leq t-1$. 
Then $\varphi_t \cha_{\la}(z;q)$ vanishes unless 
$\kappa_t(\tcore(\la)):=\mathbf{c}\in \mathcal{C}_{b;t}$ and 
$\la\supseteq\mu_{\mathbf{c}}$.
If these conditions are satisfied, then
\begin{align*}
&\varphi_t\cha_\la(z;q)\\
&\quad=\varepsilon(q)\prod_{\substack{r=0}}^{\lfloor (b-2)/2\rfloor}
\rs_{\lar{r},\lar{b-r-1}}(a+1;c_r;q)
\prod_{s=b}^{\lfloor (t+b-2)/2\rfloor}\rs_{\lar{s},\lar{t+b-s-1}}(a;c_s;q) \\
&\qquad\qquad\qquad\qquad\qquad\times\begin{cases}
1 & \text{if $b$ even, $t$ even},\\
\cha_{\lar{(b-1)/2}}(a+1;q) & \text{if $b$ odd, $t$ odd}, \\
\cha_{\lar{(t+b-1)/2}}(a;q) & \text{if $b$ even, $t$ odd},\\
\cha_{\lar{(b-1)/2}}(a+1;q)\cha_{\lar{(t+b-1)/2}}(a;q) 
& \text{if $b$ odd, $t$ even},
\end{cases}
\end{align*}
where the factor $\varepsilon(q)$ may be expressed as
\[
\varepsilon(q)=(-1)^{(\abs{\mu_{\mathbf{c}}}-(z+1)\rk(\tcore(\la))/2}
\sgn_t(\la/\mu_{\mathbf{c}})q^{\rk(\tcore(\la))}.
\]
\end{theorem}

This result contains all of the factorisation theorems for ordinary and 
universal characters previously mentioned. 
Upon setting $q=0$ all characters reduce to Schur functions (either one or
a product of two) and so we recover the straight shape case of
Theorem~\ref{Thm_skewSchur} which was stated as
Theorem~\ref{Thm_Littlewoodvarphi} in the introduction.
However, Theorem~\ref{Thm_skewSchur} is a key ingredient in our proof below,
so we are not able to claim a new proof of this result.
Substituting $q=(-1)^z$ and then choosing $z=0,1$ or $-1$ gives the
factorisations for the classical characters in Theorems~\ref{Thm_so},
\ref{Thm_o} and \ref{Thm_sp} respectively.
If we instead keep the parameter $q$ then we further obtain $q$-deformations of
these factorisations.

Our proof is based on the skew Schur expansion of $\cha_\la(z;q)$
\eqref{Eq_chizSkew}.
This is in contrast to previous proofs of these characters factorisations which
were all based on determinantal expressions.
Our technique gives a better understanding of the structure of
these factorisations. 
In particular, through Theorem~\ref{Thm_zAsym} and its corollaries, it 
explains the combinatorial mechanism of these results. 
Of course, by using the determinantal forms of all the symmetric functions 
involved it is possible to give a purely determinantal proof, however again
the sign will not be so easily expressed in this case.

\subsection{Proof of Theorem~\ref{Thm_chiz}}\label{Sec_proof}
Since the proof has several components, we break it up into smaller sections.
The initial step is obvious: apply the $t$-th Verschiebung operator 
to $\cha_\la(z;q)$ using the skew Schur expansion
\eqref{Eq_chizSkew} and Theorem~\ref{Thm_skewSchur}.
This gives
\begin{align}
\varphi_t\cha_\la(z;q)&=
\sum_{\mu\in\mathscr{P}_z}(-1)^{(\abs{\mu}-\rk(\mu)(z+1))/2}q^{\rk(\mu)}
\varphi_t s_{\la/\mu}\notag \\
&=\sum_{\substack{\mu\in\mathscr{P}_z\\\text{$\la/\mu$ $t$-tileable}}}
(-1)^{(\abs{\mu}-\rk(\mu)(z+1))/2}q^{\rk(\mu)}\sgn_t(\la/\mu)
\prod_{r=0}^{t-1}s_{\lar{r}/\mur{r}}.\label{Eq_TileableSum}
\end{align}

\subsubsection{Vanishing}
From here the vanishing part of the theorem is already evident. 
Firstly, $\la/\mu$ is $t$-tileable only if $\tcore(\la)=\tcore(\mu)$,
so that $\kappa_t(\tcore(\la))$ must lie in $\mathcal{C}_{b;t}$
since $\mu$ is $z$-asymmetric.
If this is the case then Corollary~\ref{Cor_minimal} provides the minimal
term in the sum. In the case $z\geq 0$, this term can only appear if, for 
$0\leq r\leq b-1$ with $c_r>0$ we have $\lar{r}\supseteq((a+1)^{c_r})$
and for $b\leq s\leq t-1$ with $c_s>0$ we have
$\lar{s}\supseteq(a^{c_s})$.
If $z<0$ then we need only conjugate these conditions.

\subsubsection{Identification of the prefactor}
Now assume that we are in the case where $\varphi_t\cha_\la(z)$ is nonzero. 
That is,
$\kappa_t(\tcore(\la))\in\mathcal{C}_{b;t}$ and we have the minimal requirements
on the $t$-quotient just given.
Observe that
\[
\sgn_t(\la/\mu)=\sgn_t(\la/\mu_{\mathbf{c}})\sgn_t(\mu/\mu_{\mathbf{c}}),
\]
so we may already pull out an overall sign of
$\sgn_t(\la/\mu_{\mathbf{c}})$.
Also, $\rk(\mu_{\mathbf{c}})=\rk(\tcore(\la))$ thanks to 
Lemma~\ref{Lem_rank}. The Littlewood decomposition implies
that $\abs{\mu_{\mathbf{c}}}$ is the minimal size of all partitions in the sum, 
so we in fact can remove an overall factor of 
\[
\varepsilon(q):=(-1)^{(\abs{\mu_{\mathbf{c}}}-(z+1)\rk(\tcore(\la)))/2}
q^{\rk(\tcore(\la))}\sgn_t(\la/\mu_{\mathbf{c}}),
\]
as desired.

Collecting the above we now have that
\begin{multline}
\varphi_t\cha_\la(z;q)\\
=\varepsilon(q)
\sum_{\substack{\mu\in\mathscr{P}_z\\\text{$\la/\mu$ $t$-tileable}}}
(-1)^{\sum_{r=0}^{t-1}(t\abs{\mur{r}}-(z+1)\rk_{c_r}(\mur{r}))/2}
\sgn_t(\mu/\mu_{\mathbf{c}})
\prod_{r=0}^{t-1}q^{\rk_{c_r}(\mur{r})}s_{\lar{r}/\mur{r}}.
\end{multline}
As a direct consequence of our Theorem~\ref{Thm_zAsym} we can replace the
sum over $\mu\in\mathscr{P}_z$ with a sum over $t$-tuples of partitions
satisfying the conditions \eqref{Eq_folding2} such that
\[
\mu=\phi_t^{-1}\big(\mathbf{c},(\mur{0},\dots,\mur{t-1})\big).
\]
In fact, the conditions \eqref{Eq_folding2} ensure that the product of
skew Schur functions coincides with the product obtained by expanding
the right-hand side of the theorem.

\subsubsection{Factorisation of the sign}
The only thing needed in order to show that the sum \eqref{Eq_TileableSum} 
decouples in the desired way is the factorisation of the interior sign.
This will be achieved by an inductive argument by considering 
terms in the sum, say $\mu$ and $\nu$, for which $\abs{\mu}-\abs{\nu}$ is 
as small as possible.
Also, it is most convenient here to assume that $z\geq 0$.
For $z\leq0$ the same set of steps will yield the factorisation of the sign.

Consider the case where $t+b$ is odd and $b<t-1$ and fix all entries in the
quotient of $\mu$ except for $\mur{(t+b-1)/2}$.
Since $c_{(t+b-1)/2}=0$ 
the minimal choice of quotient entry is
$\mur{(t+b-1)/2}=\varnothing$ and 
$\rk_{c_{(t+b-1)/2}}(\mur{(t+b-1)/2})=\rk(\mur{(t+b-1)/2})$.
There are two ways to add cells to $\mur{(t+b-1)/2)}$ whilst remaining in
the set of $z$-asymmetric partitions: (i) we may add a row of $a+1$ cells, the 
left-most of which sits on the main diagonal of $\mur{(t+z-1)/2}$ or (ii) a 
pair of cells at either end of a principal hook of $\mu$.
In case (i), in terms of the $t$-Maya diagram, this corresponds to moving a 
bead directly to the left of the origin $a+1$ spaces to the right.
As we know from Lemma~\ref{Lem_Sgn}, the sign $\sgn_t(\mu/\mu_{\mathbf{c}})$
will change by the number of beads passed over.
The conditions \eqref{Eq_bFold} ensure that there are no beads present in
this region in any of the runners labelled $0\leq r\leq b-1$.
Therefore the only beads counted when computing the sign lie above runner
$(t+b-1)/2$ in the column directly to the left of the origin, and strictly
between runners $b-1$ and $(t+b-1)/2$ in column $a+1$.
However, conditions \eqref{Eq_tFold2} tell us that the number of such beads
is always $(t-b-1)/2$, since the runners either side of $(t+b-1)/2$ form 
pairs up to $a$-shifted conjugation.
This introduces a factor of $(-1)^{(t-b-1)/2}$, but since we have added $a+1$
cells to the $t$-quotient and the rank has further increased by one this
sign change cancels with that coming from the exponent of $-1$, leading to
no overall sign change.
In case (ii) the rank is unchanged and the two ribbons must be 
conjugates of one another, so their heights sum to $t-1$.
Putting this together, we see that the sign associated to $\mur{(t+b-1)/2}$
is equal to $(-1)^{(\abs{\mur{t+b-1}}-(a+1)\rk(\mur{t+b-1}))/2}$.

Now assume that $b$ is odd. Again since $c_{b-1}=0$ the minimal choice of
$\mur{b-1}$ is $\varnothing$ and $\rk_{c_{b-1}}(\mur{b-1})=\rk(\mur{b-1})$. 
The analysis is almost exactly the same as that of the previous paragraph. 
We again have two cases corresponding to either increasing the rank of 
$\mur{b-1}$ or not. If we do not, then the sign will change since we add a
pair of conjugate ribbons.
If the rank does increase, then we are moving a bead from column $-1$ of 
runner $b-1$ to column $a+2$.
In the range $b\leq s\leq t-1$ we will find precisely $t-b$ beads, since
each runner will have a single bead in either column $-1$ or column $a+1$ by
the conjugation conditions.
Similar to the previous case we will find $(b-1)/2$ beads in the range
$0\leq r\leq b-1$, thus contributing $(-1)^{(2t-b-1)/2}$ all together.
But this is precisely the sign coming from the change in size and rank of the
quotient, leading to no overall change in sign.
Thus we have shown that the sign in this case changes by
$(-1)^{(\abs{\mur{b-1}}-(a+2)\rk(\mur{b-1}))/2}$.

For the next case take a pair of runners $r$ and $t+b-r-1$ for 
$b\leq r\leq t-1$ such that $r\neq t+b-r-1$.
The partitions $\mur{r}$ and $\mur{t+b-r-1}$ in the quotient are governed by a 
single partition, $\xir{r}$, such that $
\mur{r}=\xir{r}+(a^{c_r+\rk_{c_r}(\xir{r})})$ and 
$\mur{t+z-r-1}=(\xir{r})'+(a^{\rk_{c_r}(\xir{r})})$.
Without loss of generality assume that $c_r\geq 0$.
By the definition of the quotient partitions we have 
$\rk_{c_r}(\mur{r})=\rk_{c_r}(\xir{r})=\rk_{-c_r}(\xir{t+z-r-1})
=\rk_{c_{t+b-r-1}}(\mur{t+b-r-1})$.
The minimal partition in the sum, $\mu_{\mathbf{c}}$, has already absorbed 
some of the contribution from $\mur{r}$, so we are left with the sign 
contribution
\(
(-1)^{\abs{\xir{r}}-(z+1)\rk_{c_r}(\xir{r})}.
\)
As above there are two cases: (i) $\rk_{c_r}(\xir{r})$ does not increase
and (ii) $\rk_{c_r}(\xir{r})$ increases.
In case (i) then the analysis is exactly the same as before and the two ribbons
added will be conjugates of one another so that the overall sign changes.
For case (ii), it is convenient to use the $t$-Maya diagram.
Indeed, increasing $\rk_{c_r}(\nur{r})$ by one corresponds to the moving of
two beads on runners $r$ and $t+b-r-1$ from column $-1$ to column $a+1$.
If we interpret the sign of these two ribbons in terms of bead-counting,
then the only beads not double-counted are those strictly between the
runners $r$ and $t+b-r-1$.
However, between the two runners in question all quotient elements are 
$a$-shifted conjugate pairs with the addition of an $a$-symmetric partition in 
the case $t+b$ is odd.
This implies that the number of beads contributing to the sign
is equal to the number of such runners, namely to $t+b-2r-2$.
This procedure is exemplified in Figure~\ref{Fig_conjMove}.
Since we have added a single cell to $\xir{r}$ and increased its rank by $1$ 
the overall sign changes in this case.
In either case we see that the sign may be expressed as $(-1)^{\abs{\xir{r}}}$.

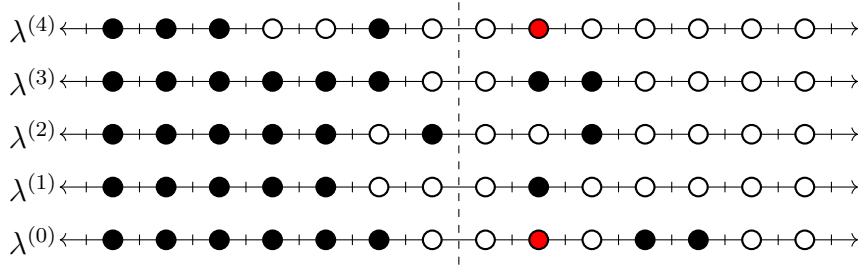
\begin{figure}[htb]
\centering
\begin{tikzpicture}[scale=0.7]
\foreach \i in {0,-1,-2,-3,-4}{
\draw[<->,opacity=0.7] (-7.5,\i) -- (7.5,\i);}
\foreach \i in {-7,...,-1,1,2,3,4,5,6,7}{
\foreach \j in {0,-1,-2,-3,-4}{
\draw[opacity=0.5] (\i,\j+0.1) -- (\i,\j-0.1);}}
\draw[opacity=0.5,dashed] (0,0.5) -- (0,-4.5);
\foreach \i in {-4,-3,-1,0,2,3,4,5,6}{
\draw[thick,fill=white] (\i+0.5,0) circle (5pt);}
\foreach \i in {-7,-6,-5,-2}{
\draw[thick,fill=black] (\i+0.5,0) circle (5pt);}
\foreach \i in {-1,0,3,4,5,6}{
\draw[thick,fill=white] (\i+0.5,-1) circle (5pt);}
\foreach \i in {-7,-6,-5,-4,-3,-2,1,2}{
\draw[thick,fill=black] (\i+0.5,-1) circle (5pt);}
\foreach \i in {-2,0,1,3,4,5,6}{
\draw[thick,fill=white] (\i+0.5,-2) circle (5pt);}
\foreach \i in {-7,-6,-5,-4,-3,-1,2}{
\draw[thick,fill=black] (\i+0.5,-2) circle (5pt);}
\foreach \i in {-2,-1,0,2,3,4,5,6}{
\draw[thick,fill=white] (\i+0.5,-3) circle (5pt);}
\foreach \i in {-7,-6,-5,-4,-3,1}{
\draw[thick,fill=black] (\i+0.5,-3) circle (5pt);}
\foreach \i in {-1,0,2,5,6}{
\draw[thick,fill=white] (\i+0.5,-4) circle (5pt);}
\foreach \i in {-7,-6,-5,-4,-3,-2,3,4}{
\draw[thick,fill=black] (\i+0.5,-4) circle (5pt);}
\draw[thick,fill=red] (1.5,0) circle (5pt);
\draw[thick,fill=red] (1.5,-4) circle (5pt);
\node at (-8,-4) {$\lar{0}$}; \node at (-8,-3) {$\lar{1}$};
\node at (-8,-2) {$\lar{2}$}; \node at (-8,-1) {$\lar{3}$};
\node at (-8,0) {$\lar{4}$};
\end{tikzpicture}
\caption{The $5$-Maya diagram of the $5$-asymmetric partition
$\la=(20~15~13~12~9~8~6~5\mid 15~10~8~7~4~3~1~0)$
with $\kappa_5(\la)=(2,-1,0,1,-2)$. The beads shaded red have been moved two 
spaces to the right, producing a sign of $-1$.}
\label{Fig_conjMove}
\end{figure}

For our final cases we take the pair of runners $r$ and $b-r-1$ where
$0\leq r\leq b-1$ and $r\neq b-r-1$.
Without loss of generality again assume that $c_r\geq 0$.
The associated pair of partitions is here governed by a single partition
$\nur{r}$ for which $\mur{r}=\nur{r}+((a+1)^{c_r+\rk_{c_r}(\nur{r})})$ and
$\mur{z-r-1}=(\nur{r})'+((a+1)^{\rk_{c_r}(\nur{r})})$.
However, the analysis of the previous paragraph applies in the same manner
to this case.
If we add a cell to $\nur{r}$ such that the $c_r$-shifted rank does not 
change then this corresponds to
a pair of conjugate ribbons and again giving an overall sign of $-1$.
On the other hand, if $\rk_{c_r}(\nur{r})$ increases then the sign also
changes by $-1$, corresponding to a total sign change of $(-1)^{\abs{\nur{r}}}$
in either case.

Combining all of the above cases we have shown that if $z\geq 0$ the sign in 
the sum is equal to
\begin{multline}
\prod_{r=0}^{\lfloor(b-2)/2\rfloor}(-1)^{\abs{\nur{r}}}
\prod_{r=b}^{\lfloor(b+t-2)/2\rfloor}(-1)^{\abs{\xir{r}}}\\
\times
\begin{cases}
1 & \text{$b$ even, $t$ even}, \\
(-1)^{(\abs{\mur{(b-1)/2}}-(a+2)\rk(\mur{(b-1)/2})/2} & \text{$b$ odd, $t$ odd}, \\
(-1)^{(\abs{\mur{(t+b-1)/2}}-(a+1)\rk(\mur{(t+b-1)/2}))/2} & 
\text{$b$ even, $t$ odd}, \\
(-1)^{(\abs{\mur{(b-1)/2}}-(a+2)\rk(\mur{(b-1)/2})+\abs{\mur{(t+b-1)/2}}-(a+1)\rk(\mur{(t+b-1)/2}))/2} & \text{$b$ odd, $t$ even}.
\end{cases}\label{Eq_signProof}
\end{multline}
It follows from the same set of steps that in the case $z\leq 0$ the same sign
is obtained, and we spare the reader repeating the details.

\subsubsection{Final steps for factorisation}
To conclude the proof, note that the structure of the sign decomposition
\eqref{Eq_signProof} is the same as that of the theorem.
In particular, the sign factors completely over the quotient, and the 
sum now decouples into a product of sums. 
Recalling our convention regarding $\rs_{\la,\mu}(a;c;q)$ when $c<0$,
the sums governed by the $\nur{r}$ for $1\leq r\leq \lfloor(b-2)/2\rfloor$ 
will each produce a copy of 
$\rs_{\lar{r},\lar{b-r-1}}(a+1;c_r;q)$.
The sums governed by the $\xir{r}$ for
$b\leq r\leq \lfloor(t+b-2)/2\rfloor$ will give copies of 
$\rs_{\lar{r},\lar{t+b-r-1}}(a;c_r;q)$.
If $b$ is odd then we also pick up a copy of 
$\cha_{\lar{(b-1)/2}}(a+1;q)$, and if $b+t$ is odd then we pick up a copy of
$\cha_{\lar{(t+b-1)/2}}(a;q)$, as desired.

\section{Plethysm rules for universal characters}\label{Sec_SXP}
As we discussed in the introduction, the factorisation of the Schur function
under $\varphi_t$ is intimately related with the Schur expansion of the
plethysm $s_\la\circ p_t$.
This expression, known as the SXP rule, has several extensions, the most 
general of which we will reproduce here together with a short proof showing
the equivalence with (the full) Theorem~\ref{Thm_skewSchur}.
Then our attention turns to generalisations of this rule to the universal
characters due to Lecouvey.

\subsection{Wildon's SXP rule}
One of the first applications of Littlewood's core and quotient construction is
to the plethysm $s_\la\circ p_t$, his expression for which is now referred to 
as the \emf{SXP rule} \cite[p.~351]{Littlewood51}.
The rule was reproved by Chen, Garsia and Remmel in \cite{CGR84}, relying on 
the $\mu=\varnothing$ 
case of Theorem~\ref{Thm_skewSchur}. It was later given an involutive proof by
Remmel and Shimozono \cite[\S5]{RS98}.
Recently, Wildon proved an extension of the SXP rule which manifests as the 
Schur expansion for the expression 
$s_\tau(s_{\la/\mu}\circ p_t)$  
and, moreover, his proof relies entirely on a sequence of bijections and
involutions.
Here, we wish to point out that Wildon's SXP rule is equivalent to the
full Theorem~\ref{Thm_skewSchur}.

\begin{theorem}[{\cite[Theorem~1.1]{Wildon18}}]\label{Thm_SXP}
For any integer $t\geq 2$ and partitions $\la,\mu,\tau$,
\[
s_\tau(s_{\la/\mu}\circ p_t)
=\sum_{\substack{\nu \\ \textup{$\nu/\tau$ $t$-tileable}}}
\sgn_t(\nu/\tau)c_{\nur{0}/\tau^{(0)},\dots,\nur{t-1}/\tau^{(t-1)},\mu}^{\la}
s_{\nu}.
\]
\end{theorem}
\begin{proof}
By the definition of the skew Schur functions we may express the coefficient
of $s_\nu$ in the Schur expansion of the left-hand side as
\[
\langle s_\tau(s_{\la/\mu}\circ p_t),s_\nu\rangle
=\langle s_{\la/\mu},\varphi_t s_{\nu/\tau}\rangle.
\]
Applying Theorem~\ref{Thm_skewSchur}  with
$\la/\mu\mapsto \nu/\tau$ on the right-hand side of this 
equation then shows that the above vanishes unless
$\nu/\tau$ is $t$-tileable, in which case it is given by
\begin{align*}
\sgn_t(\nu/\tau)\bigg \langle s_{\la/\mu},\prod_{r=0}^{t-1}
s_{\nur{r}/\tau^{(r)}}\bigg\rangle
&=
\sgn_t(\nu/\tau)\bigg \langle s_\la,s_\mu\prod_{r=0}^{t-1}
s_{\nur{r}/\tau^{(r)}}\bigg\rangle \\
&=\sgn_t(\nu/\tau)c_{\nur{0}/\tau^{(0)},\dots,\nur{t-1}/\tau^{(t-1)},\mu}^{\la}.
\qedhere
\end{align*}
\end{proof}

\subsection{SXP rules for universal characters}
Since, like the Schur functions, the universal characters admit nice 
factorisations under the map $\varphi_t$, it is natural to also seek
SXP-type rules for these symmetric functions.
This question has already been considered by Lecouvey, who, following
his paper \cite{Lecouvey09B}, gave analogues of the SXP rule for the
universal symplectic and orthogonal characters
\cite{Lecouvey09A}.
In this section we wish to restate these rules more explicitly by using our 
combinatorial framework.

Define coefficients $a_{\la,\nu}^{\bullet}(t)$ where
$\bullet$ is one of $\sp$, $\o$ or $\so^+$ by
\[
\o_\la\circ p_t
=\sum_{\nu} a_{\la,\nu}^{\o}(t) \o_\nu,
\quad \sp_\la\circ p_t=\sum_{\nu}a_{\la,\nu}^{\sp}(t),
\quad\text{and}\quad
\so^+_\la\circ p_t
=\sum_{\nu}a_{\la,\nu}^{\so^+}(t) \so_\nu^+.
\]

To begin, we first point out that it is not difficult to give explicit,
albeit cumbersome, expressions for these coefficients.

\begin{lemma}[{\cite[Lemma~3.1.1]{Lecouvey09B}}]\label{Lem_cumbersome}
We have
\begin{align*}
a_{\la,\nu}^{\o}(t)
&=\sum_{\mu\in\mathscr{P}_1} 
\sum_{\substack{\xi\\\tcore(\xi)=\varnothing}}\sum_{\substack{\eta\\\textup{$\eta$ even}}}
(-1)^{\abs{\mu}/2}\sgn_t(\xi)
c_{\xir{0},\dots,\xir{t-1},\mu}^\la 
c_{\nu,\eta}^{\xi}, \\
a_{\la,\nu}^{\sp}(t)
&=\sum_{\mu\in\mathscr{P}_{-1}} 
\sum_{\substack{\xi\\\tcore(\xi)=\varnothing}}\sum_{\substack{\eta\\\textup{$\eta'$ even}}}
(-1)^{\abs{\mu}/2}\sgn_t(\xi)
c_{\xir{0},\dots,\xir{t-1},\mu}^\la 
c_{\nu,\eta}^{\xi}, \\
a_{\la,\nu}^{\so^+}(t)
&=\sum_{\mu\in\mathscr{P}_0} 
\sum_{\substack{\xi\\\tcore(\xi)=\varnothing}}\sum_{\eta}
(-1)^{\abs{\nu}+(\abs{\mu}-\rk(\mu))/2}\sgn_t(\xi)
c_{\xir{0},\dots,\xir{t-1},\mu}^\la 
c_{\nu,\eta}^{\xi}.
\end{align*}
Moreover, $a_{\la,\nu}^{\o}(t)=(-1)^{\abs{\la}(t-1)}a_{\la',\nu'}^{\sp}(t)$.
\end{lemma}
\begin{proof}
We begin with the first identity.
Expanding $\o_\la\circ p_t$ in terms of skew Schur functions and then 
applying the SXP rule of Theorem~\ref{Thm_SXP} with 
$\tau=\varnothing$ leads to
\[
\o_\la\circ p_t
=\sum_{\mu\in\mathscr{P}_1} 
\sum_{\substack{\xi\\\tcore(\xi)=\varnothing}}
(-1)^{\abs{\mu}/2}\sgn_t(\xi)
c_{\xir{0},\dots,\xir{t-1},\mu}^\la s_\xi.
\]
By the character interrelation formula
\eqref{Eq_CIR} we have
\[
\o_\la\circ p_t
=
\sum_{\nu}\o_\nu\bigg(
\sum_{\mu\in\mathscr{P}_1} 
\sum_{\substack{\xi\\\tcore(\xi)=\varnothing}}\sum_{\substack{\eta\\\text{$\eta$ even}}}
(-1)^{\abs{\mu}/2}\sgn_t(\xi)
c_{\xir{0},\dots,\xir{t-1},\mu}^\la 
c_{\nu,\eta}^{\xi}\bigg).
\]
The same steps yield the formulae for the other characters.
For the duality between the coefficients one uses the involution $\omega$
combined with \eqref{Eq_omegapt}.
Note that the universal characters are not homogeneous symmetric functions.
However,
the skew Schur expansions show that in the symplectic and even orthogonal
cases they are sums of homogeneous symmetric
functions whose degrees agree modulo two, and so the identity still holds in
this case.
\end{proof}

In fact, Lecouvey shows that $a_{\la,\nu}^{\o}(t)=a_{\la,\nu}^{\so^+}(t)$,
his argument being based on the fact that $\pi_n'$ and the plethysm by
$p_t$ commute.
Using this fact applied to the universal character $\o_\la$ shows the 
equality of the coefficients for $n\geq t l(\la)$.
We have not found a simple explanation at the level of universal characters 
for why the expressions given above for the coefficients 
$a_{\la,\nu}^{\o}(t)$ and $a_{\la,\nu}^{\so^+}(t)$ coincide.

As we remarked in Subsection~\ref{Sec_subClass}, Lecouvey has given algorithms
for computing the action of $\varphi_t$ on classical group characters.
For the odd orthogonal group $\mathrm{SO}_{2n+1}$ this algorithm is
crucial in stating his SXP-type rules. 
In view of Theorem~\ref{Thm_so}, we may restate this algorithm entirely in 
terms of the classical Littlewood decomposition.
What follows is a reinterpretation of the algorithm given in 
\cite[\S4]{Lecouvey09A}.

\begin{const}\label{Const_SO}
Let $n,t\in\mathbb{N}$ be such that $n=at+b$ for $0\leq b\leq t-1$.
Further let $\la$ be a partition of length at most $n$ with 
$\kappa_t(\la)=(c_0,\dots,c_{t-1})$ and 
quotient $(\lar{0},\dots,\lar{t-1})$.
Reading indices modulo $t$ we define for 
$0\leq r\leq\lfloor\frac{t-2}{2}\rfloor$ sequences
\[
\gamma^{(r)}:=[\lar{-r-b-1},\lar{r-b}]_{2a+d_r}
+(c_{-r-b-1}^{2a+d_r}),
\]
where we additionally set
\[
d_r:=\begin{cases}
1 & \text{if $0\leq r\leq b-1$}, \\
2 & \text{if $0\leq t-r-1\leq b-1$ and $0\leq r\leq b-1$},\\
0 & \text{otherwise}.
\end{cases}
\]
Also, if $t$ is odd, $\gamma^{((t-1)/2)}:=\lar{(t-1)/2-b}$ where
$l(\gamma^{((t-1)/2)})\leq a+d_{(t-1)/2}$ and
\[
d_{(t-1)/2}:=\begin{cases} 1 & \text{if $b>(t-1)/2$}, \\
0 & \text{if $b\leq (t-1)/2$}.\end{cases}
\]
Given the above, write
\[
\gamma_n(\la;t)
:=(\gamma^{(0)},\dots,\gamma^{(\lfloor(t-1)/2\rfloor)}).
\]
\end{const}

The output of this construction is a dominant weight for
\[
G_n(\la;t):=
\mathrm{GL}_{2a+d_0}\times\cdots\times
\mathrm{GL}_{2a+d_{\lfloor(t-2)/2\rfloor}}
\times\begin{cases} \mathrm{SO}_{2(a+d_{(t-1)/2})+1}
& \text{if $t$ odd}, \\
1 & \text{if $t$ even},
\end{cases}
\]
a Levi subgroup of $\mathrm{SO}_{2n+1}$.
Let $\mathfrak{g}_n(\la;t)$ denote the corresponding Lie algebra.
We write $V^{\mathfrak{so}_{2n+1}}(\la)$ for the irreducible finite-dimensional
representation of $\mathrm{SO}_{2n+1}$ of highest weight $\la$,
and similarly for $V^{\mathfrak{g}_n(\mu;t)}(\gamma(\mu;t))$.
The branching coefficient 
$[V^{\mathfrak{so}_{2n+1}}(\la):V^{\mathfrak{g}_n(\mu;t)}(\gamma_n(\mu;t))]$
then gives the multiplicity of $V^{\mathfrak{g}_n(\mu;t)}(\gamma_n(\mu;t))$
when $V^{\mathfrak{so}_{2n+1}}(\la)$ is restricted to $G_n(\mu;t)$.
Note that if $b=0$, so that $n$ is a multiple of $t$, then this construction
will output the partitions in the quotient paired as in
Theorem~\ref{Thm_so}.

For an example, take $(n,t)=(8,5)$ so that $(a,b)=(1,3)$.
Then for the partition 
$\la=(15,14,10,7,4,3,2,1)$ we have $\kappa_5(\la)=(0,-1,1,0,0)$ and
\[
\big(\lar{0},\lar{1},\lar{2},\lar{3},\lar{4}\big)
=\big(\varnothing,\varnothing,(2,2,1),(1),(3,1)\big).
\]
Construction~\ref{Const_SO} will output
\[
\gamma_8(\la;5)=\big((0,-1,-1), (0,0,-1),(3,1)\big),
\]
and $G_8(\la;5)=\mathrm{GL}_3\times \mathrm{GL}_3\times\mathrm{SO}_5$.

Let $\mathscr{C}_{b;t}$ denote the set of sequences 
$(c_0,\dots,c_{t-1})\in \mathbb{Z}^t$ such that
$c_{r-b}+c_{t-r-1-b}=0$, where indices are read modulo $t$.
When viewed as encoding $t$-cores, this corresponds to the set of $t$-cores 
$\mu$ which, after shifting the indices of $\kappa_t(\mu)$ cyclicly $b$ places
to the right, are self-conjugate.
We are now ready to state the SXP rule for the odd orthogonal characters.

\begin{theorem}[{\cite[Theorem~4.1.1]{Lecouvey09A}}]\label{Thm_LecFrob}
Let $t\geq 2$ and $n$ be integers such that $n=at+b$ where
$0\leq b\leq t-1$. Then for any partition $\la$ with $l(\la)\leq n$,
\begin{multline}\label{Eq_LecFrob}
\pi_n(\so_\la^+)\circ p_t \\
=\sum_{\substack{\mu\\l(\mu)\leq n\\\kappa_t(\mu)\in\mathscr{C}_{b;t}}}
(-1)^{(\abs{\tilde{\mu}}-\rk(\tilde{\mu}))/2}\sgn_t(\mu/\tilde{\mu})
[V^{\mathfrak{so}_{2n+1}}(\la):V^{\mathfrak{g}_n(\mu;t)}(\gamma_n(\mu;t))]
\pi_n(\so_\mu^+).
\end{multline}
\end{theorem}

There are also versions of this result for $\mathrm{Sp}_{2n}$ and 
$\mathrm{O}_{2n}$ in the case $t$ is even, but they are not stated in
\cite{Lecouvey09A}.
For $t$ odd there cannot be rules of this form since the coefficients describing
the action of the Verschiebung operator on the characters are not branching
coefficients. This is further clarified by the appearance of the ``negative''
odd orthogonal characters in Theorems~\ref{Thm_so} and \ref{Thm_o}.
However, Theorem~\ref{Thm_LecFrob} is all that is needed to
state the universal character lifts of these rules.

As remarked in \cite[p.~769]{Lecouvey09A}, it is possible to give explicit
expressions for the branching coefficients occurring in \eqref{Eq_LecFrob} in 
terms of (multi-)Littlewood--Richardson coefficients. 
These are particularly simple for $n\geq tl(\la)$, since for these values of 
$n$ the coefficients stabilise. 

\begin{lemma}\label{Lem_coefs}
Assume that $n\geq tl(\la)$ and $\tcore(\la)=\varnothing$. If $t$ is even, then
\[
[V^{\mathfrak{so}_{2n+1}}(\la):V^{\mathfrak{g}_n(\mu;t)}(\gamma_n(\mu;t))]
=\sum_{\eta^1,\dots,\eta^{t/2}}
c_{\eta^1,\dots,\eta^{t/2},\mur{0},\dots,\mur{t-1}}^\la.
\]
If $t$ is odd, then\footnote{It is correct that $\eta^{(t+1)/2}$ occurs 
twice in the lower-index of the multi-Littlewood--Richardson coefficient.}
\[
[V^{\mathfrak{so}_{2n+1}}(\la):V^{\mathfrak{g}_n(\mu;t)}(\gamma_n(\mu;t))]
=\sum_{\eta^1,\dots,\eta^{(t+1)/2}}
c_{\eta^1,\dots,\eta^{(t+1)/2},\eta^{(t+1)/2},\mur{0},\dots,\mur{t-1}}^\la.
\]
Else, if $\tcore(\la)\neq \varnothing$ then 
\[
[V^{\mathfrak{so}_{2n+1}}(\la):V^{\mathfrak{g}_n(\mu;t)}(\gamma_n(\mu;t))]=0.
\]
\end{lemma}
\begin{proof}
Assume that $n\geq tl(\la)$ and $\tcore(\mu)=\varnothing$.
The output of Construction~\ref{Const_SO} applied to $\mu$ yields a tuple 
of weights $(\gamma^{(0)},\dots,\gamma^{(\lfloor(t-1)/2\rfloor}))$
which are made up of pairs of partitions, with an additional 
single partition if $t$ is odd.
If $t$ is even then we first use the restriction rule of Theorem~\ref{Thm_SOGL},
which is positive since $n\geq tl(\la)\geq 2l(\la)$.
From here we then iterate the rule of Theorem~\ref{Thm_GLGLGL} to branch onto
the group $G_n(\mu;t)$.
In the case $t$ is odd, then we begin with the rule of Theorem~\ref{Thm_SOGLSO},
choosing $k=a+d_{(t-1)/2}$, and then iterate Theorem~\ref{Thm_GLGLGL} 
to land in $G_n(\mu;t)$.
Since we have assumed that $n\geq tl(\la)$, these rules will all contain only
positive terms, expressed as sums of multi-Littlewood--Richardson coefficients
as in the statement.

Now assume that $\tcore(\mu)\neq \varnothing$ and that $n=tl(\la)$.
We have that $\sum_{r=0}^{t-1}l(\mur{r})\leq l(\mu)\leq l(\la)$, which may be
seen from the $t$-Maya diagram.
Since $\mu$ has nonempty $t$-core there exists some $r$ for which 
$c_{t-r-1}\neq 0$ and $l(\mur{r})+l(\mur{t-r-1})\leq l(\la)$. 
This means that the length of at least one of the partitions
which make up $\gamma^{(r)}$, which has been shifted by $c_{t-r-1}$, 
will be greater than the length of $\la$, and so the branching coefficients 
will vanish in this case.
\end{proof}

Note that the above also shows that for any $n$ such that $n\geq tl(\la)$ the
branching coefficients are always the same, since increasing $n$ by one
merely permutes the $\mur{r}$.

Let us denote the stablised version of the above coefficients from
Lemma~\ref{Lem_coefs} by $b_{\la,\mu}(t)$.
We may now state the SXP rules for the universal characters.
\begin{theorem}[{\cite[Theorem~4.5.1]{Lecouvey09A}}]
For $\la$ a partition and $t\geq 2$ and integer we have
\begin{align*}
\so_\la^+\circ p_t
&=\sum_{\substack{\mu\\\tcore(\mu)=\varnothing}}
\sgn_t(\mu)b_{\la,\mu}(t)\so_\mu^+,\\
\o_\la\circ p_t
&=\sum_{\substack{\mu\\\tcore(\mu)=\varnothing}}
\sgn_t(\mu)
b_{\la,\mu}(t)\o_\mu, \\
\intertext{and}
\sp_\la\circ p_t
&=(-1)^{\abs{\la}(t-1)}\sum_{\substack{\mu\\\tcore(\mu)=\varnothing}}
\sgn_t(\mu')b_{\la',\mu'}(t) \sp_\mu.
\end{align*}
where $b_{\la,\mu}(t)$ denotes the branching coefficients of
Lemma~\ref{Lem_coefs}.
\end{theorem}
\begin{proof}
The first equation is immediate from the large-$n$ vanishing part of 
Lemma~\ref{Lem_coefs}.
As remarked after the proof of Lemma~\ref{Lem_cumbersome} the coefficients in
the expansions of $\so^+_\la\circ p_t$ and $\o_\la\circ p_t$ coincide, which
establishes the second equality.
By the duality part of that same lemma, or by directly applying the 
$\omega$ involution,
\[
a_{\la,\mu}^{\sp}(t)=(-1)^{\abs{\la}(t-1)}a_{\la',\mu'}^{\o}(t)
=(-1)^{\abs{\la}(t-1)}\sgn_t(\mu')b_{\la',\mu'}(t).\qedhere
\]
\end{proof}

As this section shows, SXP rules for symplectic and orthogonal characters 
are intimately connected with the representation theory of their associated
groups. Thus, it is not clear if there exists a general SXP rule
for the symmetric function $\cha_\la(z;q)$ in the same manner.
We have also not found a simple proof of the fact that the stabilised 
coefficients $\sgn_t(\mu)b_{\la,\mu}(t)$ agree with $a_{\la,\mu}^{\so^+}(t)$
as expressed in Lemma~\ref{Lem_cumbersome}.
Finally, it does not appear that adjoint relation between $\varphi_t$ and the
plethysm by $p_t$ may be employed to give short proofs of the SXP rules based
on the factorisations of Theorems~\ref{Thm_so}, \ref{Thm_o} and
\ref{Thm_sp}. 
This is because there is no orthonormality for the universal characters under
the Hall inner product. In contrast, Lecouvey uses deformations of the 
Verschiebung operator with respect to the standard inner product on the 
character rings under which the Weyl characters are orthonormal.

\section{Variations on factorisations}\label{Sec_disc}

To conclude, we explain the connections between the results of this paper and
very closely related results: symmetric functions twisted by roots of unity
and characters of the symmetric group.

\subsection{Symmetric polynomials twisted by roots of unity}
A perspective we have not taken in this paper is that of ``twisting'' a
symmetric polynomial by a primitive $t$-th root of unity.
In fact, this is very closely connected to the original work of Littlewood
and Richardson on this topic; see the papers
\cite{Littlewood35,LR34a,LR34b} or Littlewood's book \cite[\S7.3]{Littlewood40}.
The interested reader should consult the recent paper of 
Ayyer and Kumari \cite{AK25}, which proves new results regarding twisting
both ordinary and universal characters by roots of unity, as well as 
surveying some of the results we will now discuss.

A simple generating function argument shows that the action of the $t$-th 
Verschiebung operator on, for instance, the complete homogeneous symmetric
functions, agrees with the result of replacing
$X_n\mapsto (X_n,\xi X_n,\dots,\xi^{t-1} X_n)$ where
$a X_n:=(ax_1,\dots,ax_n)$ for any $a\in\mathbb{C}$ and evaluating.
Littlewood and Richardson apply this twisting to the bialternant formula for 
the Schur functions and then through a sequence of matrix manipulations 
deduce the vanishing and factorisation.
This is the same approach which is taken in the work of Ayyer and Kumari 
\cite{AK22}.
The advantage of this approach is it allows for slightly more general 
statements, such as the following theorem due to Littlewood and Richardson.

\begin{theorem}[{\cite[Theorem~XI]{LR35}}]\label{Thm_y}
Let $\la$ be a partition of length at most $nt+1$. Then for another variable
$y$ we have that
\[
s_\la(X_n,\xi X_n,\dots, \xi^{t-1} X_n,y)=0
\]
unless $\tcore(\la)=(k)$ for some $0\leq k\leq t-1$, in which case
\[
s_\la(X_n,\xi X_n,\dots, \xi^{t-1} X_n,y)
=\sgn_t(\la/(k))y^{k}s_{\lar{k-1}}(X_n^t,y^t)\prod_{\substack{r=0\\r\neq k-1}}^{t-1} s_{\lar{r}}(X_n^t).
\]
\end{theorem}

This has itself been generalised in several directions.
For instance, Littlewood also characterises the vanishing and factorisation of
$s_\la(1,\dots,\xi^{m})$ where $m$ is an arbitrary positive integer independent
of $t$ \cite[\S7.2]{Littlewood40} which has proved important in the context of
cyclic sieving; see \cite[Theorem~4.3]{RSW04},
\cite[Lemma~6.2]{Rhoades10} and \cite[Theorem~4.4]{Pfannerer22}.
Recently Kumari extended this by replacing the variable $y$ in 
Theorem~\ref{Thm_y} by a set of variables $y_1,\dots,y_r$,
generalising Littlewood's result \cite{Kumari24}.
She also proves similar results for the characters of the symplectic and 
orthogonal groups for these same twists, however, the evaluations are not
always products, and are quite complicated.
None of these results extend elegantly through the Verschiebung operators.
We have given a version of Theorem~\ref{Thm_y} involving a deformation of
the Verschiebung operator \cite[Proposition~6.2]{Albion23}, but this is, in our
opinion, not particularly natural.
There is also a version of Theorem~\ref{Thm_skewSchur}
for flagged skew Schur functions \cite{Kumar23}.

Outside of the realm of classical symmetric functions and classical 
group characters there has been little interest in the action of the 
operators $\varphi_t$.
To our knowledge the only work in this direction is due to Mizukawa
\cite{Mizukawa02}, who has given expressions for the action of the Verschiebung
operators on the Schur $Q$-functions, as well as SXP-type rules.
These involve variants of the Littlewood decomposition for partitions
with distinct parts (also called \emf{bar partitions}), the concepts of which 
were developed in the papers of Morris \cite{Morris65} and Olsson 
\cite{Olsson87}.
By considering the double of a strict partition which is $1$-asymmetric, an 
idea Humphreys attributes to Macdonald \cite{Humphreys86}, these results may
be phrased in terms of the ordinary Littlewood decomposition.
Using this, one may extend Mizukawa's results to skew Schur $Q$-functions
by use of their definition as a Pfaffian, which plays the same role as the
Jacobi--Trudi formula in the proof of Theorem~\ref{Thm_skewSchur}.

\subsection{Characters of the symmetric group}\label{Sec_chars}
In this paper we have not discussed equivalent statements for the characters 
of the symmetric group, as in Littlewood and Richardson's original 
Theorem~\ref{Thm_LittlewoodMult}.
Here we give the precise connection between these two perspectives.

The following may be found in, for instance, \cite[\S I.7]{Macdonald95}.
Let $R^n$ denote the space of class functions on $\mathfrak{S}_n$.
The \emf{characteristic map} $\ch^n:R^n\longrightarrow\La^n$ is defined by
\[
\ch^n(f)
=\frac{1}{n!}\sum_{w\in\mathfrak{S}_n}f(w)p_{\mathrm{cyc}(w)},
\]
where $\La^n$ denotes the space of homogeneous symmetric functions of degree
$n$ and $\mathrm{cyc}(w)$ is a partition of $n$ encoding the cycle type of
$w$.
Under this map $\ch^n(\chi^\la)=s_\la$.
Now let $R:=\bigoplus_{n\geq 0} R^n$. 
For $f\in R^n$ and $g\in R^m$ defining the induction product
$f\cdot g
:=\Ind_{\mathfrak{S}_n\times \mathfrak{S}_m}^{\mathfrak{S}_{n+m}}(f\otimes g)$
turns $R$ into a graded algebra.
We also have a scalar product on $R$ which for $f=\sum_{n\geq 0} f_n$ and
$g=\sum_{n\geq 0} g_n$ is given by
\begin{equation}\label{Eq_Rprod}
\langle f,g\rangle'
:=\sum_{n\geq0}\langle f_n,g_n\rangle_{\mathfrak{S}_n},
\end{equation}
where $\langle f_n,g_n\rangle_{\mathfrak{S}_n}$ is the ordinary scalar product
of $\mathfrak{S}_n$ characters.
The map $\ch:=\bigoplus_{n\geq0}\ch^n$ is then an isometric isomorphism between
$R$ and $\La$. 
We now define the actions of the $t$-th Verschiebung operator and 
its adjoint on $R$. 
On $\La$ this adjoint is the plethysm by a power sum $p_t$, but in general it
is the \emf{Frobenius operator} or \emf{Adams operation} (the former is not to 
be confused with the Frobenius characteristic, another name given to $\ch$).
As in the case of the characteristic map we first define
for $f\in R^n$ the operator $\varphi_t^n$ by
\begin{equation}\label{Eq_phiR}
\varphi_t^n(f)(\mu)=f(t\mu).
\end{equation}
From this we see that if $f\in R^n$ then $\varphi_t(f)\in R^{n/t}$ if $t\mid n$
and is the zero function otherwise. 
Then $\varphi_t:=\bigoplus_{n\geq0} \varphi_t^n$.
In particular if $1_n$ denotes the trivial representation of $\mathfrak{S}_n$
then $\varphi_t(1_n)=1_{n/t}$ if $t$ divides $n$ and is equal to zero otherwise.
Since $\ch(1_n)=h_n$ it follows that $\ch\varphi_t=\varphi_t\ch$, where on the
left we use \eqref{Eq_phiR} and on the right we use the Verschiebung operator on
$\La$. The same is true of $\ch^{-1}$.
The action of $\psi_t^n$ may now be defined by
\begin{equation}\label{Eq_psitR}
\psi_t^n(\chi^\la)(\mu)=
\begin{cases}
t^{l(\mu)}\chi^\la(\mu/t) & \text{if $t\mid \mu_i$ for all $i\geq 1$},\\
0 & \text{otherwise}.
\end{cases}
\end{equation}
Then also set $\psi_t:=\bigoplus_{n\geq0}\psi_t^n$.
Note similarity between \eqref{Eq_psitR} and the expression
for $\varphi_t p_\la$ from Proposition~\ref{Prop_hep}.
The fact that these operators are adjoint with respect to \eqref{Eq_Rprod}
then follows from the orthonormality of the irreducible characters.
All in all, the point of the above constructions is that the characteristic map,
when applied to the expression of Theorem~\ref{Thm_LittlewoodMult}, yields
the expression for $\varphi_t s_\la$ of Theorem~\ref{Thm_Littlewoodvarphi}.
By applying $\ch^{-1}$ to Theorem~\ref{Thm_skewSchur} we obtain the following
theorem of Farahat.
\begin{theorem}[{\cite{Farahat54}}]\label{Thm_farahat}
Let $\la/\mu$ be a skew shape with $\abs{\la/\mu}=nt$. Then 
$\varphi_t(\chi^{\la/\mu})=0$ unless $\la/\mu$ is $t$-tileable, in which case
\[
\varphi_t(\chi^{\la/\mu})
=\sgn_t(\la/\mu)\prod_{r=0}^{t-1}\chi^{\lar{r}/\mur{r}},
\]
where the product on the right-hand side is the induction product.
\end{theorem}

The operators $\varphi_t$ and $\psi_t$ in the context of symmetric group
characters first appeared in the relatively unknown paper of Kerber,
S\"anger and Wagner \cite{KSW81}.
In particular, they state the actions \eqref{Eq_phiR} and \eqref{Eq_psitR}
in Section~4 of that paper, together with the adjoint relation.
These are then used to give a proof of Farahat's generalisation of
Theorem~\ref{Thm_LittlewoodMult}, which describes the action of
the Verschiebung operator on the skew character $\chi^{\la/\mu}$.
This is notably different to Farahat's proof, which uses symmetric functions.
They also prove the character-theoretic analogue of the SXP rule, 
our Theorem~\ref{Thm_SXPLittlewood}, which is equivalent to Littlewood's 
original rule for the plethysm $s_\la\circ p_t$.
Another proof of Farahat's theorem is given in \cite[\S3]{EPW14}.
For a more recent application of these ideas to characters of the symmetric 
group see the paper of Rhoades \cite{Rhoades22}.

There has also been some recent interest in the character values
$\chi^\la_{t\mu}$ from a slightly different perspective.
L\"ubeck and Prasad \cite{LP21} have shown that for $\la$ a partition 
with empty $2$-core the character value $\chi^\la_{2\mu}$ is equal,
up to the sign $\sgn_2(\la)$, to the 
value of an irreducible character of the wreath product 
$\mathbb{Z}_2\wr \mathfrak{S}_n$ (also known as the hyperoctahedral group) 
indexed by $(\lar{0},\lar{1})$ 
evaluated at the conjugacy class $(\mu,\varnothing)$.
(For the necessary background on characters of wreath products see \cite[Chapter~I, Appendix~B]{Macdonald95}.)
Their proof is heavily algebraic, and along the way they prove and apply 
the $t=2$ cases of Theorems~\ref{Thm_LittlewoodMult} and 
\ref{Thm_Littlewoodvarphi}.
They also consider the case where $\xcore{2}(\la)=(1)$, which itself hinges on
the $t=2$ case of 
Theorem~\ref{Thm_y} and its character-theoretic analogue, also 
contained in a theorem of Littlewood \cite[p.~340]{Littlewood51}.
This was generalised by Adin and Roichman \cite{AR22}, who further show that
for $\tcore(\la)=\varnothing$
the value $\sgn_t(\la)\chi^\la_{t\mu}$ may be expressed as the character of 
the wreath product $G\wr \mathfrak{S}_n$ indexed by $(\lar{0},\dots,\lar{t-1})$ 
evaluated at the $t$-tuple $(\mu,\varnothing,\dots,\varnothing)$ where $G$ is 
any finite abelian group of order $t$.
Their proof is of a more combinatorial flavour, using Stembridge's 
Murnaghan--Nakayama rule for wreath products \cite[\S4]{Stembridge89} and ribbon
combinatorics. 
Note that this does not cover the vanishing of the character values
$\chi^\la(t\mu)$ in the case $\tcore(\la)$ is nonempty.
Since Stembridge's rules work more generally for skew shapes, it would be 
interesting to investigate a skew extension of these results, putting 
Farahat's theorem into the picture.
For further remarks on this side of the story we refer to the review of the
paper of L\"ubeck and Prasad by Wildon \cite{Wildon22}, which includes a
proof of Theorem~\ref{Thm_LittlewoodMult} using the SXP rule.

\subsection*{Acknowledgements}
I thank C\'edric Lecouvey for kindly explaining the connections between our
results, his work and the results of Ayyer and Kumari.
Thanks also to Arvind Ayyer for useful discussions.

\end{document}